\documentclass[10pt,reqno]{amsart}
\usepackage{amsmath,amssymb,graphicx,mathrsfs,amsopn,stmaryrd,amsthm,color,bm,times}
\usepackage{ytableau}
\usepackage[colorlinks=true,citecolor=blue,linkcolor=blue,pagebackref]{hyperref}
\usepackage[final,inline]{showlabels}
\usepackage[english]{babel}
\usepackage[mathscr]{euscript}
\usepackage{xcolor}
\usepackage{geometry}
\usepackage{tikz}    
\usetikzlibrary{arrows,matrix,decorations.markings,shapes.geometric}   
\usetikzlibrary{decorations.pathreplacing}
\newcommand{\tikznode}[3][inner sep=0pt]{\tikz[remember
picture,baseline=(#2.base)]{\node(#2)[#1]{$#3$};}}
\numberwithin{equation}{section}
\newtheorem{theorem}{Theorem}[section]
\newtheorem{proposition}[theorem]{Proposition}
\newtheorem{lemma}[theorem]{Lemma}
\newtheorem{corollary}[theorem]{Corollary}
\newtheorem{conj}[theorem]{Conjecture}
\theoremstyle{definition} 
\newtheorem{eg}[theorem]{Example}

\theoremstyle{remark}
\newtheorem{rem}[theorem]{Remark}

\newcommand{\beq}{\begin{equation}}
\newcommand{\eeq}{\end{equation}}
\newcommand{\be}{\begin{equation*}}
\newcommand{\ee}{\end{equation*}}
\newcommand{\bs}{\boldsymbol}

\newcommand{\C}{\mathbb{C}}

\newcommand{\Z}{\mathbb{Z}}

\newcommand{\mc}{\mathcal}
\newcommand{\cD}{\mathcal{D}}

\newcommand{\cT}{\mathcal{T}}
\newcommand{\cP}{\mathcal{P}}

\newcommand{\g}{\mathfrak{g}}
\newcommand{\gl}{\mathfrak{gl}}
\newcommand{\h}{\mathfrak{h}}

\newcommand{\fkS}{\mathfrak{S}}

\newcommand{\End}{\mathrm{End}}

\newcommand{\str}{{\mathrm{str}}} 

\newcommand{\tl}{\tilde}
\newcommand{\gge}{\geqslant}
\newcommand{\lle}{\leqslant}
\newcommand{\la}{\lambda}

\newcommand{\glMN}{\mathfrak{gl}_{m|n}}
\newcommand{\UglMN}{\mathrm{U}(\mathfrak{gl}_{m|n})}

\newcommand{\YglMN}{\mathrm{Y}(\mathfrak{gl}_{m|n})}
\newcommand{\YglN}{\mathrm{Y}(\mathfrak{gl}_{N})}

\newcommand{\YslMN}{\mathrm{Y}(\mathfrak{sl}_{m|n})}

\newcommand{\bmx}{\begin{pmatrix}}    
\newcommand{\emx}{\end{pmatrix}}

\begin{document}
\pagestyle{myheadings}
\setcounter{page}{1}

\title[Jacobi-Trudi identity and super Yangian]{Jacobi-Trudi identity and Drinfeld functor for super Yangian}

\author{Kang Lu and Evgeny Mukhin}
\address{K.L.: Department of Mathematics, University of Denver, 
\newline
\strut\kern\parindent 2390 S. York St., Denver, CO 80208, USA}\email{kang.lu@du.edu}
\address{E.M.: Department of Mathematical Sciences,
Indiana University-Purdue University\newline
\strut\kern\parindent Indianapolis, 402 N.Blackford St., LD 270,
Indianapolis, IN 46202, USA}\email{emukhin@iupui.edu}

\begin{abstract} 
We show that the quantum Berezinian which gives a generating function of the integrals of motions of XXX spin chains associated to super Yangian $\YglMN$ can be written as a ratio of two difference operators of orders $m$ and $n$ whose coefficients are ratios of transfer matrices corresponding to explicit skew Young diagrams.

In the process, we develop several missing parts of the representation theory of $\YglMN$ such as $q$-character theory, Jacobi-Trudi identity, Drinfeld functor, extended T-systems, Harish-Chandra map.
		
		\medskip 
		
\end{abstract}

\maketitle


\thispagestyle{empty}
\section{Introduction}	
This paper deals with the representation theory of $\YglMN$, the Yangian associated with the general Lie superalgebra $\glMN$. This theory has been actively developed in recent years, see \cite{Gow05,Gow07,Peng,Tsy,Zhh16,Zhh18}. Many results have resemblance to the even case and are proved by similar methods. However, there is a number of important and non-trivial differences.

In fact, we have an additional important motivation for this study which comes from the theory of integrable systems. The super Yangian $\YglMN$ contains a commutative subalgebra (called the \textit{Bethe subalgebra}) given by the expansion of a certain quantum Berezinian, see \cite{MR14,Naz} and \eqref{ber}.
The coefficients of the expansion are transfer matrices related to representations associated to single-column Young diagrams, see \eqref{eq expand}. The study of the spectra of transfer matrices is the central topic in the theory of integrable systems for the last half of the century. The main method used is called the Bethe ansatz which allows to construct eigenvectors from solutions of a system of algebraic equations called the Bethe ansatz equations, see e.g. \cite{BR08,HLM,LM19,MTV06,Tsu:1997}.

Besides an eigenvector, to each solution of the Bethe ansatz equations one associates a fundamental object, called an ``oper". The set of opers is easier to study compared to the set of solutions.  In the case of the super Yangian $\YglMN$, the opers are written as a ratio of two difference scalar operators of orders $m$ and $n$, \cite{HLM,HMVY,LM19}. Based on the experience with even cases, it is natural to expect the existence of the ``universal oper" with coefficients in the Bethe subalgebra which produces the scalar opers when the coefficients are replaced by their eigenfunctions on a given eigenvector.

That leads to a conjecture that the quantum Berezinian can be compactly written in the form $\mathfrak D_1\mathfrak D_2^{-1}$ where $\mathfrak D_1$ and $\mathfrak D_2$ are difference operators of orders $m$ and $n$,  respectively, with coefficients in the Bethe subalgebra. We show that this is indeed the case and, moreover, that the $i$-th coefficients of $\mathfrak D_1$ and $\mathfrak D_2$ are given by transfer matrices corresponding to skew Young diagrams given on Figure \ref{young fig}
\begin{figure}[!htb]
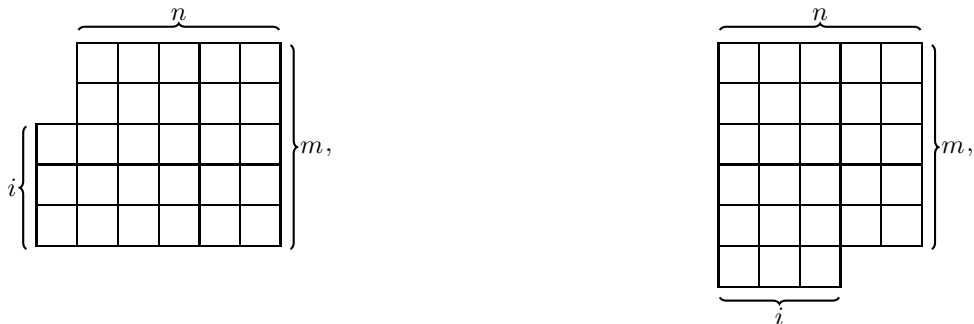
 \label{young fig}
    \centering
    \begin{minipage}{.5\textwidth}
        \centering
$$
\ytableausetup{nosmalltableaux,centertableaux}
\begin{ytableau}
       \none  & \tikznode{e4}{~} & ~ & ~ & ~ & \tikznode{e1}{~} \\
       \none  & ~ & ~ & ~ & ~ & ~ \\
       \tikznode{e2}{~}  & ~ & ~ & ~ & ~ & ~ \\
       ~ & ~ & ~ & ~ & ~ & ~ \\
       \tikznode{e3}{~}  & ~ & ~ & ~ & ~ & \tikznode{e5}{~}
\end{ytableau}\quad \ \ \, ,
$$
\tikz[overlay,remember picture]{%
\draw[decorate,decoration={brace},thick] 
([yshift=5mm,xshift=-2mm]e4.north west) -- 
([yshift=5mm,xshift=2mm]e1.north east) node[midway,above]{$n$};

\draw[decorate,decoration={brace},thick] 
([yshift=3mm,xshift=3.5mm]e1.north east) -- 
([yshift=-2mm,xshift=3.5mm]e5.south east) node[midway,right]{$m$};

\draw[decorate,decoration={brace},thick] 
([yshift=-2mm,xshift=-3.5mm]e3.south west) -- 
([yshift=3mm,xshift=-3.5mm]e2.north west) node[midway,left]{$i$};
}
    \end{minipage}%
    \begin{minipage}{0.5\textwidth}
        \centering
$$
\begin{ytableau}
       \none  & \none & \none & \none & \none \\
        \tikznode{s4}{~} & ~ & ~ & ~ & \tikznode{s1}{~} \\
       ~ & ~ & ~ & ~ & ~ \\
        ~ & ~ & ~ & ~ & ~ \\
       ~ & ~ & ~ & ~ & ~ \\
       ~ & ~ & ~ & ~ & \tikznode{s5}{~} \\
       \tikznode{s3}{~}  & ~ & \tikznode{s2}{~}  & \none & \none
\end{ytableau}\quad \ \ \, ,
$$
\tikz[overlay,remember picture]{%
\draw[decorate,decoration={brace},thick] 
([yshift=5mm,xshift=-2mm]s4.north west) -- 
([yshift=5mm,xshift=2mm]s1.north east) node[midway,above]{$n$};

\draw[decorate,decoration={brace},thick] 
([yshift=3mm,xshift=3.5mm]s1.north east) -- 
([yshift=-2mm,xshift=3.5mm]s5.south east) node[midway,right]{$m$};

\draw[decorate,decoration={brace},thick] 
([yshift=-3mm,xshift=2mm]s2.south east) -- 
([yshift=-3mm,xshift=-2mm]s3.south west) node[midway,below]{$i$};
}
    \end{minipage}
    \caption{Skew Young diagrams corresponding to the coefficients of the difference operators.} 
\end{figure}
divided by the transfer matrix corresponding to the $m\times n$ rectangle, see Theorem \ref{thm ratio}. 

An important goal is to understand the kernels of
operators $\mathfrak D_1$ and $\mathfrak D_2$. The two kernels are joined together to form a superspace consisting of rational functions, which is called a $\glMN$ space, \cite{HMVY}. Our results make a step towards understanding the $\glMN$ spaces, see the discussion in Section \ref{sec spectra}.

\medskip

The proof of the above result is obtained by a direct computation with the use of the Jacobi-Trudi identity for the $\YglMN$-modules related to the participating skew Young diagrams.
Thus a part of this paper is used to prepare the necessary background and develop all the missing parts.

More generally, we study irreducible $\YglMN$-modules $L(\la/\mu)$ corresponding to arbitrary skew Young diagrams, called skew representations. Skew representations of the Yangian and quantum affine algebra of type A in the even case were introduced in \cite{Che87b,Che89} and studied in \cite{MY12a,Naz04,NT98,NT02}. We establish the formulas for the $q$-characters (that is the joint generalized eigenvalues of the Gelfand-Tsetlin subalgebra) of such modules in terms of the semi-standard Young tableaux, see Theorem \ref{thm character}, and prove the Jacobi-Trudi identity for the $q$-characters. The ways to attack the Jacobi-Trudi identity are well-known. Here we use the Lindstr\"{o}m-Gessel-Viennot method with the appropriate modifications, see Theorem \ref{thm Jacobi-Trudi}.

We take the opportunity to define the Drinfeld functor which constructs a $\YglMN$-module from a representation of the degenerate affine Hecke algebra $\mc H_l$. The key fact is that the Jacobi-Trudi formula does not depend on $m$ and $n$. It implies that the $\YglMN$-module $L(\la/\mu)$ comes from the same $\mc H_l$-module for all $m$ and $n$.  Since the Drinfeld functor is exact and an equivalence of categories in the even case, see \cite{CP,Dri86} or Theorem \ref{thm equivalence}, we obtain a tool to translate the information about representations of even Yangian $\YglN$ to super Yangian $\YglMN$. 

We give a couple of such examples, describing sufficient conditions for tensor products of evaluation $\YglMN$-modules to be irreducible, see Theorem \ref{thm irr suff cond} and establishing the extended T-systems, see Corollary \ref{cor T super}. 

Section \ref{sec: Drinfeld} is dedicated solely to the Drinfeld functor and its applications, it is not used in the proof of our main Theorem \ref{thm ratio}. In our opinion the findings of this section have their own merits and we plan to use them in our future work.

\medskip

The paper is constructed as follows. We start by organizing known facts about the general Lie superalgebra $\glMN$ and the super Yangian $\YglMN$ in Section \ref{sec:super yangian}. In addition we compute some information about the coproduct, see Proposition \ref{prop:cop}, which allows us to introduce the $q$-character ring homomorphism in Section \ref{sec:q-char}. Section \ref{sec skew} is devoted to the study of $\YglMN$-modules related to skew Young diagrams. The Drinfeld functor is defined and studied in Section \ref{sec: Drinfeld}. The applications in the form of irreducibility conditions for tensor products and the extended T-systems are given in Sections \ref{sec:irr} and \ref{sec:T}, respectively.
Section \ref{sec transfer matrix} deals with transfer matrices and, in particular, with quantum Berezinians. We study the Harish-Chandra map which connects the transfer matrices to $q$-characters, see Section \ref{sec:HC}. The main result of this section is Theorem \ref{thm ratio}.

\medskip

{\bf Acknowledgments.} We thank V. Tarasov for stimulating discussions. This work was partially supported by a grant from the Simons Foundation \#353831.

\section{Super Yangian $\YglMN$}\label{sec:super yangian}
\subsection{Lie superalgebra $\glMN$}\label{sec glmn}Throughout the paper, we work over $\C$. In this section, we recall the basics of the Lie superalgebra $\glMN$, see e.g. \cite{CW12} for more detail. We simply write $\gl_m$ for $\gl_{m|0}$.

A \emph{vector superspace} $W = W_{\bar 0}\oplus W_{\bar 1}$ is a $\Z_2$-graded vector space. We call elements of $W_{\bar 0}$ \emph{even} and elements of
$W_{\bar 1}$ \emph{odd}. We write $|w|\in\{\bar 0,\bar 1\}$ for the parity of a homogeneous element $w\in W$. Set $(-1)^{\bar 0}=1$ and $(-1)^{\bar 1}=-1$.

Fix $m,n\in \Z_{\gge 0}$. Set $I:=\{1,2,\dots,m+n-1\}$ and $\bar I:=\{1,2,\dots,m+n\}$. We also set $|i|=\bar 0$ for $1\lle i\lle m$ and $|i|=\bar 1$ for $m< i\lle m+n$. Define $s_i=(-1)^{|i|}$ for $i\in \bar I$.

The Lie superalgebra $\glMN$ is generated by elements $e_{ij}$, $i,j\in \bar I$, with the supercommutator relations
\[
[e_{ij},e_{kl}]=\delta_{jk}e_{il}-(-1)^{(|i|+|j|)(|k|+|l|)}\delta_{il}e_{kj},
\]
where the parity of $e_{ij}$ is $|i|+|j|$. Denote by $\UglMN$ the universal enveloping superalgebra of $\glMN$. The superalgebra $\UglMN$ is a Hopf superalgebra with the coproduct given by $\Delta(x)=1\otimes x+x\otimes 1$ for all $x\in \glMN$. 

The \emph{Cartan subalgebra $\h$} of $\glMN$ is spanned by $e_{ii}$, $i\in \bar I$. Let $\epsilon_i$, $i\in \bar I$, be a basis of $\h^*$ (the dual space of $\h$) such that $\epsilon_i(e_{jj})=\delta_{ij}$. There is a bilinear form $(\ ,\ )$ on $\h^*$ given by $(\epsilon_i,\epsilon_j)=s_i\delta_{ij}$. The \emph{root system $\bf{\Phi}$} is a subset of $\h^*$ given by
\[
{\bf \Phi}:=\{\epsilon_i-\epsilon_j~|~i,j\in \bar I \text{ and }i\ne j\}.
\]
We call a root $\epsilon_i-\epsilon_j$ \emph{even} (resp. \emph{odd}) if $|i|=|j|$ (resp. $|i|\ne |j|$).

Set $\alpha_i:=\epsilon_i-\epsilon_{i+1}$ for $i\in I$. Denote by ${\bf P}:=\oplus_{i\in \bar I}\Z \epsilon_i$, ${\bf Q}:=\oplus_{i\in I}\Z \alpha_i$, and ${\bf Q}_{\gge 0}:=\oplus_{i\in I}\Z_{\gge 0} \alpha_i$ the \emph{weight lattice}, the \emph{root lattice}, and the \emph{cone of positive roots}, respectively. Define a partial ordering $\gge$ on $\h^*$: $\alpha\gge \beta$ if $\alpha-\beta\in {\bf Q}_{\gge 0}$. There is a natural $\Z_2$-grading on ${\bf P}$ such that the parity $p_{\alpha}$ of $\alpha\in {\bf P}$ is given by
\beq\label{eq:weight-parity}
p_{\alpha}:=\sum_{i\in \bar I}(\alpha,\epsilon_i)|i|.
\eeq

A module $M$ over a superalgebra $\mathcal A$ is a vector superspace $M$ with a homomorphism of superalgebras $\mathcal A\to \End(M)$. A $\glMN$-module is a module over $\mathrm{U}(\glMN)$.

Let $\alpha\in {\bf P}$. We call a nonzero vector $v$ in a $\glMN$-module $M$ a \emph{singular vector of weight} $\alpha$ if $v$ satisfies
\[
e_{ii}v=\alpha(e_{ii})v,\quad e_{jk}v=0,
\]
for $i\in \bar I$ and $1\lle j< k\lle m+n$. Denote by $L^p(\alpha)$ the irreducible $\glMN$-module generated by a 
singular vector of weight $\alpha$ and of parity $p$. We simply write $L(\alpha)$ for $L^{p_{\alpha}}(\alpha)$.

For a $\glMN$-module $M$, define the \emph{weight subspace of weight} $\alpha$ by
\[
(M)_{\alpha}:=\{v\in M\ |\ e_{ii}v=\alpha(e_{ii})v,\ i\in \bar I\}.
\]
If $(M)_{\alpha}\ne 0$, we call $\alpha$ a {\it weight} of $M$. Denote by $\mathrm{wt}(M)$ the set of all weights of $M$. If $v\in (M)_{\alpha}$ and $v$ is non-zero, then we write $\mathrm{wt}(v)=\alpha$. We focus on the $\glMN$-modules $M$ such that $(M)_{\alpha}=0$ unless $\alpha\in\bf P$. We say that $M$ is \emph{$\bf P$-graded}. We call a vector $v\in M$ \emph{singular} if $e_{ij}v=0$ for $1\lle i<j\lle m+n$.

Let $V:=\C^{m|n}$ be the vector superspace with a basis $v_i$, $i\in \bar I$, such that $|v_i|=|i|$. Let $E_{ij}\in\End(V)$ be the linear operators such that $E_{ij}v_k=\delta_{jk}v_i$. The map $\rho_V:\glMN\to \End(V),\ e_{ij}\mapsto E_{ij}$ defines a $\glMN$-module structure on $V$. As a $\glMN$-module, $V$ is isomorphic to $L({\epsilon_1})$. The vector $v_i$ has weight $\epsilon_i$. The highest weight vector is $v_1$ and the lowest weight vector is $v_{m+n}$. We call it the \emph{vector representation} of $\glMN$.

Let $\gl_{n|m}$ be the Lie superalgebra defined in the same way as $\glMN$ with $m$ and $n$ interchanged. There exists a Lie superalgebra isomorphism between $\gl_{m|n}$ and $\gl_{n|m}$ given by the map 
\beq\label{eq change mn}
\varsigma_{m|n}: e_{ij}\mapsto e_{m+n+1-i,m+n+1-j}.
\eeq

Fix $m',n'\in\Z_{\gge 0}$ and consider the Lie superalgebra $\gl_{m'|n'}$. For this algebra we also choose the standard parity, namely, we set $|i|=\bar 0$ if and only if $1\lle i\lle m'$.

We also consider a larger Lie superalgebra $\gl_{m'+m|n'+n}$. 
For  $\gl_{m'+m|n'+n}$, we fix the parity by
\beq\label{eq parity}
|i|=\begin{cases}
\bar 0, &\text{ if }1\lle i\lle m' \text{ or } m'+n'+1\lle i\lle m'+n'+m,\\
\bar 1, &\text{ if }m'+1\lle i\lle m'+n' \text{ or } m'+n'+m+1\lle i \lle m'+n'+m+n.
\end{cases}  
\eeq
Clearly, we have the embeddings of Lie superalgebras given by
\[
\gl_{m'|n'}\to \gl_{m'+m|n'+n},\quad e_{ij}\mapsto e_{ij},
\]
\[
\gl_{m|n}\to \gl_{m'+m|n'+n},\quad e_{ij}\mapsto e_{m'+n'+i,m'+n'+j}.
\]

Define the \emph{supertrace} $\str:\End(\C^{m|n})\to \C$,
\[
\str(E_{ij})=s_j\delta_{ij}.
\]
The supertrace is supercyclic, that is 
$$\str([E_{ij},E_{rs}])=0.$$
Here $[\cdot,\cdot]$ is the supercommutator of linear operators.

Denote by $\mathfrak{sl}_{m|n}$ the Lie subalgebra of $\gl_{m|n}$ consisting of all elements acting on $V$ as matrices with zero supertrace.

Define the \emph{supertranspositions} $t$ and $\top$,
\beq\label{eq trans}
t:\End(V)\to \End(V),\quad E_{ij}^t= (-1)^{|i||j|+|j|}E_{ji},
\eeq
\beq\label{eq trans 2}
\top:\End(V)\to \End(V),\quad E_{ij}^\top= (-1)^{|i||j|+|i|}E_{ji}.
\eeq
Both supertranspositions are anti-homomorphisms and respect the supertrace,
\beq\label{eq transpose cyclic}
(AB)^*=(-1)^{|A||B|}B^*A^*,\qquad \str(A)=\str(A^*), 
\eeq
for all $(m+n)\times (m+n)$ matrices $A$ and $B$, where $*$ is either $t$ or $\top$. We also have $t=\top^3$ and $t^4=\top^4=1$.

\subsection{Hook partitions, skew Young diagrams, and polynomial modules}\label{sec:hook-p}
Let $\la=(\la_1\gge \la_2\gge \cdots)$ be a partition of $\ell$: $\la_i\in\Z_{\gge 0}$, $\la_i=0$ if $i\gg 0$, and $|\la|:=\sum_{i=1}^\infty \la_i=\ell$. We denote by $\la'$ the \emph{conjugate} of the partition $\la$. The number $\la_1'$ is the length of the partition $\la$, namely the number of nonzero parts of $\la$. Let $\mu=(\mu_1\gge \mu_2\gge \cdots)$ be another partition such that $\mu_i\lle \la_i$ for all $i=1,2,\dots$. Consider the \emph{skew Young diagram} $\la/ \mu$ which is defined as the set of pairs
\[
\{(i,j)\in\Z^2~|~i\gge 1,\ \la_i\gge j > \mu_i\}.
\]
When $\mu$ is the zero partition, then $\la/ \mu$ is the usual Young diagram corresponding to $\la$.

We use the standard representation of skew Young diagrams on the coordinate plane $\mathbb R^2$ with coordinates $(x,y)$. Here we use the convention that $x$ increases from north to south while $y$ increases from west to east. Moreover, the pair $(i,j)\in \la/ \mu$ is represented by the unit box whose south-eastern corner has coordinate $(i,j)\in\mathbb Z^2$. We also define the \emph{content} of the box corresponding to $(i,j)\in \la/ \mu$ by $c(i,j)=j-i$. 

A \textit{semi-standard Young tableau} of shape $\la/\mu$ is the skew Young diagram $\la/\mu$ with an element from $\{1,2,\dots,m+n\}$ inserted in each box such that the following conditions are satisfied:
\begin{enumerate}
    \item the numbers in boxes are weakly increasing along rows and columns;
    \item the numbers from $\{1,2,\dots,m\}$ are strictly increasing along columns;
    \item the numbers from $\{m+1,m+2,\dots,m+n\}$ are strictly increasing along rows.
\end{enumerate}
For a semi-standard Young tableau $\mathcal T$ of shape $\la/\mu$, denote by $\mathcal T(i,j)$ the number in the box representing the pair $(i,j)\in \la/\mu$. 
\begin{eg}
Let $\la=(5,3,3,3,3)$, $\mu=(3,3,2,2)$, $m=n=2$, then the skew Young diagram $\la/\mu$ has the shape as one of the following. 
$$
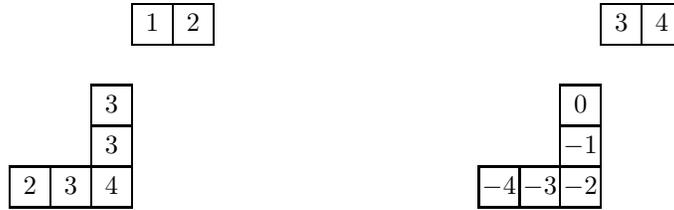
\begin{ytableau}
\none & \none & \none & 1 & 2\\
 \none & \none & \none\\
 \none & \none & 3\\
 \none & \none & 3\\
 2 & 3 &4	
\end{ytableau}
\qquad\qquad \qquad\qquad\qquad
\begin{ytableau}
\none & \none & \none & 3 & 4\\
 \none & \none & \none\\
 \none & \none & 0\\
 \none & \none & -1\\
 -4 & -3 & -2	
\end{ytableau}
$$
In the picture above, the left one is an example of a semi-standard Young tableau of shape $\la/\mu$. We have $\mc T(1,4)=1$, $\mc T(1,5)=\mc T(5,1)=2$, $\mc T(3,3)=\mc T(4,3)=\mc T(5,2)=3$, and $\mc T(5,3)=4$. In the right, we wrote in each box its content.  
\end{eg}

A \textit{standard Young tableau} of shape $\la/\mu$ is the skew Young diagram $\la/\mu$ with an element from $\{1,\dots,|\la|-|\mu|\}$ inserted in each box such that the numbers in boxes are strictly increasing along rows and columns. 

\begin{eg}
There is a distinguished standard Young tableau obtained by filling numbers along rows from left to right and top to bottom. We call it the {\it row tableau}. Similarly, one defines the {\it column tableau}. Here are the row (on the left) and the column (on the right) tableaux for the skew Young diagram $\la/\mu$ in the previous example.
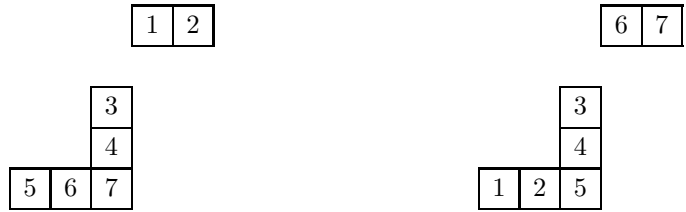
\[
\begin{ytableau}
\none & \none & \none & 1 & 2\\
 \none & \none & \none\\
 \none & \none & 3\\
 \none & \none & 4\\
 5 & 6 &7	
\end{ytableau}
\qquad\qquad \qquad\qquad\qquad
\begin{ytableau}
\none & \none & \none & 6 & 7\\
 \none & \none & \none\\
 \none & \none & 3\\
 \none & \none & 4\\
 1 & 2 & 5	
\end{ytableau}
\]
\end{eg}

Recall that $V=\C^{m|n}$ denotes the vector representation of $\gl_{m|n}$.
A $\gl_{m|n}$-module is called a \emph{polynomial module} if it is a submodule of $V^{\otimes l}$ for some $l\in\Z_{\gge 0}$. We call a partition $\la$ an $(m|n)$-{\it hook partition} if $\la_{m+1}\lle n$. Let $\mathscr P_l(m|n)$ be the set of all $(m|n)$-hook partitions of $l$ and $\mathscr P(m|n)$ the set of all $(m|n)$-hook partitions. In particular, $\mathscr P_l(m):=\mathscr P_l(m|0)$ is the set of all partitions of $l$ with length $\lle m$. It is well-known that irreducible polynomial $\gl_{m|n}$-modules are parameterized by $\mathscr P(m|n)$.

For $\la\in\mathscr P(m|n)$, define the $\gl_{m|n}$-weight $\la^\natural$ by
\beq\label{eq natural}
\la^\natural=\sum_{i=1}^m \la_i\epsilon_i+\sum_{j=1}^n \max\{\la_j'-m,0\}\epsilon_{m+j}.
\eeq
We sometimes use the notation $\la^{\natural}_{[m|n]}$ to stress the dependence of $\la^\natural$ on $m$ and $n$.

Let $\mc P\in \End(V \otimes V)$ be the super flip operator,
\[
\mathcal P=\sum_{i,j\in \bar I} s_jE_{ij}\otimes E_{ji}.
\]
Let $\mathfrak S_l$ be the symmetric group permuting $\{1,2,\dots,l\}$. The symmetric group $\mathfrak S_l$ acts naturally on $V^{\otimes l}$, where the simple transposition $\sigma_k=(k,k+1)$ acts as 
\beq\label{eq perm}
\mathcal P^{(k,k+1)}=\sum_{i,j\in \bar I} s_j E_{ij}^{(k)}E_{ji}^{(k+1)}\in \End(V^{\otimes l}),
\eeq
where we use the standard notation
\[
E_{ij}^{(k)}=1^{\otimes (k-1)}\otimes E_{ij}\otimes 1^{\otimes (l -k)}\in \End(V^{\otimes l}),\qquad k=1,\dots,l.
\]

Let $\mathcal S(\la)$ be the finite-dimensional irreducible representation of $\mathfrak S_l$ corresponding to the partition $\la$.
\begin{theorem}[Schur-Sergeev duality \cite{Ser85}]\label{thm schur}
The $\mathfrak S_l$-action and $\gl_{m|n}$-action on $V^{\otimes l}$ commute. Moreover, as a $\mathrm{U}(\glMN)\otimes \C[\mathfrak S_l]$-module, we have
\[
V^{\otimes l}\cong\bigoplus_{\la\in\mathscr P_l(m|n)}L(\la^\natural)\otimes \mathcal S(\la).
\]	
\end{theorem}
 
For $\la\in\mathscr P(m|n)$, we have $\la'\in \mathscr P(n|m)$. The $\gl_{m|n}$-module obtained by pulling back the $\gl_{n|m}$-module $L((\la')_{[n|m]}^{\natural})$ through the isomorphism $\varsigma:\glMN\to\gl_{n|m}$, see \eqref{eq change mn}, is isomorphic to $L(\la_{[m|n]}^\natural)$.

For $\la\in\mathscr P(m'+m|n'+n)$, define a $\gl_{m'+m|n'+n}$-weight $\la^\circ$ by
\begin{align}
\la^\circ=\sum_{i=1}^{m'} \la_i\epsilon_i+\sum_{j=1}^{n'} \max\{\la_j'-m',0\}\epsilon_{m'+j}+&\sum_{i=m'+1}^{m'+m} \max\{\la_i-n',0\}\epsilon_{n'+i}\nonumber\\+&\sum_{j=n'+1}^{n'+n} \max\{\la_j'-m'-m,0\}\epsilon_{m'+m+j}.\label{eq circ}
\end{align}
This definition is dictated by the choice of parity \eqref{eq parity}, see \cite{BR83}.
We will use $\la^\circ$ in Section \ref{sec skew} to define skew representations of super Yangian.

\subsection{Super Yangian $\YglMN$}\label{sec rtt}
We recall the definition of the super Yangian $\YglMN$ from \cite{Naz}. 

The super Yangian $\YglMN$ is the $\Z_2$-graded unital associative algebra over $\C$ with generators $\{t_{ij}^{(r)}\ |\ i,j\in \bar I, \, r\gge 1\}$ and defining relations
\beq\label{eq:comm-generators}
[t_{ij}^{(r)},t_{kl}^{(s)}]=(-1)^{|i||j|+|i||k|+|j||k|}\sum_{a=0}^{\min(r,s)-1}(t_{kj}^{(a)}t_{il}^{(r+s-1-a)}-t_{kj}^{(r+s-1-a)}t_{il}^{(a)}),
\eeq
where the generators $t_{ij}^{(r)}$ have parities $|i|+|j|$.

The super Yangian $\YglMN$ has the RTT presentation as follows. Define the rational R-matrix $R(u)\in \End(V \otimes  V)$ by $R(u)=1-\mathcal P/u$. The rational R-matrix satisfies the quantum Yang-Baxter equation
\beq\label{eq yang-baxter}
R_{12}(u_1-u_2)R_{13}(u_1-u_3)R_{23}(u_2-u_3)=R_{23}(u_2-u_3)R_{13}(u_1-u_3)R_{12}(u_1-u_2).
\eeq
Define the generating series
\[
t_{ij}(u)=\delta_{ij}+\sum_{k=1}^\infty t_{ij}^{(k)}u^{-k}
\]
and the operator $T(u)\in \End(V)\otimes\YglMN[[u^{-1}]] $,
\[
T(u)=\sum_{i,j\in \bar I} (-1)^{|i||j|+|j|}E_{ij}\otimes t_{ij}(u).
\]
Denote by
\beq\label{eq matrix notation}
T_k(u)=\sum_{i,j\in \bar I} (-1)^{|i||j|+|j|}E_{ij}^{(k)}\otimes t_{ij}(u)\in\End(V^{\otimes l}) \otimes\YglMN[[u^{-1}]] .
\eeq
Then defining relations \eqref{eq:comm-generators} can be written as
\[
R(u_1-u_2)T_1(u_1)T_2(u_2)=T_2(u_2)T_1(u_1)R(u_1-u_2)\in \End(V^{\otimes 2}) \otimes\YglMN[[u^{-1}]].
\]
In terms of generating series, defining relations \eqref{eq:comm-generators} are equivalent to
\beq\label{eq:comm-series}
(u_1-u_2)[t_{ij}(u_1),t_{kl}(u_2)]=(-1)^{|i||j|+|i||k|+|j||k|}(t_{kj}(u_1)t_{il}(u_2)-t_{kj}(u_2)t_{il}(u_1)).
\eeq

The super Yangian $\YglMN$ is a Hopf superalgebra with coproduct, antipode, counit given by
\beq\label{eq Hopf}
\Delta: t_{ij}(u)\mapsto \sum_{k\in \bar I} t_{ik}(u)\otimes t_{kj}(u),\qquad S: T(u)\mapsto T(u)^{-1},\qquad \varepsilon: T(u)\mapsto 1.
\eeq
Let $\Delta^{\rm op}$ be the opposite coproduct of $\YglMN$,
\beq\label{eq coop}
\Delta^{\rm op}(t_{ij}(u))=\sum_{k\in \bar I} (-1)^{(|i|+|k|)(|j|+|k|)}t_{kj}(u)\otimes t_{ik}(u).
\eeq

For $z\in\C$ there exists an isomorphism of Hopf superalgebras,
\begin{align}
&\tau_z:\YglMN\to\YglMN, && t_{ij}(u)\mapsto t_{ij}(u-z).\label{eq tau z}
\end{align}

For any $\YglMN$-module $M$, denote by $M_z$ the $\YglMN$-module obtained from pulling back $M$ through the isomorphism $\tau_z$.

The super Yangian $\YglMN$ has a weight decomposition ($\bf P$-grading) with respect to Cartan subalgebra of $\mathrm{U}(\glMN)\subset \YglMN$. The generator $t_{ij}^{(k)}$ has weight $\epsilon_i-\epsilon_j$. 

We have the standard PBW theorem.
\begin{theorem}[\cite{Gow07}]\label{thm PBW}
Fix some ordering on the generators $t_{ij}^{(k)}$, $i,j\in \bar I$ and $k\in \Z_{>0}$, for the super Yangian $\YglMN$. Then the ordered monomials of these generators, with at most power 1 for odd generators, form a basis of $\YglMN$.
\end{theorem}

Set $\mathcal B:=1+u^{-1}\C[[u^{-1}]]$. For any series $\vartheta(u)\in \mathcal B$, the map
\beq\label{eq mul auto}
\Gamma_{\vartheta}: T(u)\mapsto \vartheta(u)T(u)
\eeq
defines an automorphism of $\YglMN$. Denote by $\YslMN$ the subalgebra of $\YglMN$ which consists of all elements that are fixed under automorphisms $\Gamma_{\vartheta}$ for all $\vartheta(u)\in \mathcal B$.

Let $\mathfrak z_{m|n}$ be the center of super Yangian $\YglMN$. If $m\ne n$, then we have an isomorphism of algebras
\[
\YglMN\cong \mathfrak z_{m|n}\otimes \YslMN,
\]
see \cite[Proposition 8.1]{Gow07}.

Let $\mathrm{Y}(\gl_{n|m})$ be the super Yangian defined in the same way as $\YglMN$ by interchanging $m$ and $n$. 

Let $\eta_{m|n}$ be the automorphism of $\YglMN$ given by
\beq\label{eq iso antipode}
\eta_{m|n}: T(u)\mapsto T(-u)^{-1}.
\eeq
Define an isomorphism of superalgebras $\varrho_{m|n}:\YglMN\mapsto \mathrm{Y}(\gl_{n|m})$ by
\[
\varrho_{m|n}:t_{ij}(u)\mapsto t_{m+n+1-i,m+n+1-j}(-u).
\]
Denote by $\hat \varsigma$ the composition of isomorphisms of superalgebras
\beq\label{eq inter mn yangian}
\hat \varsigma_{m|n}=\varrho_{m|n}\circ \eta_{m|n}, \qquad \hat \varsigma_{m|n}:\YglMN\mapsto \mathrm{Y}(\gl_{n|m}).
\eeq

Finally, for fixed $m',n'\in\Z_{\gge 0}$, one also defines a larger super Yangian $\mathrm Y(\gl_{m'+m|n'+n})$ following the choice of parities as in \eqref{eq parity}. 

\subsection{Gauss decomposition}\label{sec:Gauss}
The Gauss decomposition of $\YglMN$, see \cite{Gow07,Peng}, gives generating series
\[
e_{ij}(u)=\sum_{r\gge 1}e_{ij}^{(r)}u^{-r},\quad f_{ji}(u)=\sum_{r\gge 1}f_{ji}^{(r)}u^{-r},\quad d_k(u)=1+\sum_{r\gge 1}d_{k}^{(r)}u^{-r},
\]
where $1\lle i< j\lle m+n$ and $k\in \bar I$, such that
\begin{align*}
t_{ii}(u)&=d_i(u)+\sum_{k<i}f_{ik}(u)d_k(u)e_{ki}(u),\\
t_{ij}(u)&=d_i(u)e_{ij}(u)+\sum_{k<i}f_{ik}(u)d_k(u)e_{kj}(u),\\
t_{ji}(u)&=f_{ji}(u)d_i(u)+\sum_{k<i}f_{jk}(u)d_k(u)e_{ki}(u).
\end{align*}

For $i\in I$ and $k\in \bar I$, let
\[
e_{i}(u)=e_{i,i+1}(u)=\sum_{r\gge 1}e_i^{(r)}u^{-r},\quad f_{i}(u)=f_{i+1,i}(u)=\sum_{r\gge 1}f_i^{(r)}u^{-r},
\]
\[
d_{k}'(u)=(d_k(u))^{-1}=1+\sum_{r\gge 1}d_k'^{(r)}u^{-r}.
\]
We use the convention $d_k^{(0)}=d_k'^{(0)}=1$.

The parities of $e_{ij}^{(r)}$ and $f_{ji}^{(r)}$ are the same as that of $t_{ij}^{(r)}$ while all $d_k^{(r)}$ and $d_k'^{(r)}$ are even. The super Yangian $\YglMN$ is generated by $e_i^{(r)}$, $f_i^{(r)}$, $d_k^{(r)}$, $d_k'^{(r)}$, where $i\in I$ and $k\in \bar I$, and $r\gge 1$. The full defining relations are described in \cite[Lemma 4 or Theorem 3]{Gow07}. Here we only write down the following relations. Let $\phi_i(u)=d_i'(u)d_{i+1}(u)=1+\sum_{r\gge 1}\phi_i^{(r)}u^{-r}$. Then one has $[d_{i}^{(r)}, d_{j}^{(s)}]=0$,
\begin{align}
&[d_{i}^{(r)}, e_{j}^{(s)}]=(\epsilon_i,\alpha_j) \sum_{t=0}^{r-1} d_{i}^{(t)} e_{j}^{(r+s-1-t)},\quad [d_{i}^{(r)}, f_{j}^{(s)}]=-(\epsilon_i,\alpha_j)\sum_{t=0}^{r-1} f_{j}^{(r+s-1-t)} d_{i}^{(t)}, \label{eq com de}\\
&[e_{j}^{(r)}, f_{k}^{(s)}]=-s_{j+1}\delta_{jk}\sum_{t=0}^{r+s-1} d_{j}^{\prime(t)} d_{j+1}^{(r+s-1-t)}=-s_{j+1}\delta_{jk}\phi_j^{(r+s-1)}.\label{eq com ef}
\end{align}
Moreover, the subalgebra $\YslMN$ is generated by the coefficients of the series $\phi_i(u)$, $e_i(u)$, $f_i(u)$ for $i\in I$.

Let $\mathrm{Y}_{m|n}^+$, $\mathrm{Y}_{m|n}^-$, and $\mathrm{Y}_{m|n}^0$ be the subalgebras of $\YglMN$ generated by coefficients of the series $e_i(u)$, $f_i(u)$, and $d_j(u)$, respectively. It is known from \cite[proof of Theorem 3]{Gow07} that 
$$
\YglMN\cong \mathrm{Y}_{m|n}^-\otimes \mathrm{Y}_{m|n}^0 \otimes \mathrm{Y}_{m|n}^+
$$ 
as vector spaces and $d_i^{(r)}$ are algebraically free generators of $\mathrm{Y}_{m|n}^0$.

The Gauss decomposition for super Yangian associated to non-standard parity sequences is studied in \cite{Peng}. In particular, one obtains generating series $e_i(u)$, $f_i(u)$, $d_i(u)$, $d_{i}^{\prime}(u)$ and generators $e_i^{(r)}$, $f_i^{(r)}$, $d_i^{(r)}$, $d_{i}^{\prime(r)}$ for $\mathrm{Y}(\gl_{n|m})$ with standard parities and for $\mathrm{Y}(\gl_{m'+m|n'+n})$ with parities in \eqref{eq parity}. We refer the reader to \cite{Peng,Tsy} for the explicit relations of $\mathrm{Y}(\gl_{m'+m|n'+n})$ in these generators. 

We conclude this section with the following lemma used in Section \ref{sec div}.

\begin{lemma}[{\cite[Proposition 4.2]{Gow07}}]\label{lem iso mn inter}
For the isomorphism $\hat\varsigma_{m|n}:\YglMN\mapsto \mathrm{Y}(\gl_{n|m})$ defined in \eqref{eq inter mn yangian}, we have
\[
\hat\varsigma_{m|n}: d_i(u)\mapsto (d_{m+n+1-i}(u))^{-1},\ e_{j}(u)\mapsto -f_{m+n-j}(u),\, f_j(u)\mapsto -e_{m+n-j}(u),
\]
for $i\in \bar I$ and $j\in I$.
\end{lemma}

\subsection{Highest and lowest $\ell$-weight representations}\label{sec high rep}
Recall $\mathcal B=1+u^{-1}\C[[u^{-1}]]$ and set $\mathfrak B:=\mathcal B^{\bar I}\times \Z_2$.
We call an element $\bm\zeta\in \mathfrak B$ an \emph{$\ell$-weight}. We write $\ell$-weights in the form $\bm\zeta=(\zeta_i(u))^{p(\bs \zeta)}_{i\in \bar I}$, where $p(\bm\zeta)\in \Z_2$ and $\zeta_i(u)\in \mathcal B$ for all $i\in \bar I$.

Clearly $\mathfrak B$ is an abelian group with respect to the point-wise multiplication of the tuples and the addition of the parities. Let $\Z[\mathfrak B]$ be the group ring of $\mathfrak B$ whose elements are finite $\Z$-linear combinations of the form $\sum a_{\bm \zeta}[\bm\zeta]$, where $a_{\bm\zeta}\in \Z$.

Let $M$ be a $\YglMN$-module. We say that a nonzero $\Z_2$-homogeneous vector $v\in M$ is \emph{of $\ell$-weight $\bm \zeta$} if $d_{i}(u)v=\zeta_i(u)v$ for $i\in \bar I$ and the parity of $v$ is given by $p(\bs\zeta)$. We say that a vector $v\in M$ of $\ell$-weight $\bs\zeta$ is a \emph{highest (resp. lowest) $\ell$-weight vector of $\ell$-weight $\bs\zeta$} if $e_{ij}(u)v=0$ (resp. $f_{ji}(u)v=0$) for all $1\lle i< j\lle m+n$. The module $M$ is called a \emph{highest (resp. lowest) $\ell$-weight module of $\ell$-weight $\bm \zeta$} if $M$ is generated by a highest (resp. lowest) $\ell$-weight vector of $\ell$-weight $\bm \zeta$.

We say that two highest $\ell$-weight modules $M_1, M_2$ are \emph{isomorphic up to a parity} if either $M_1\cong M_2$ or 
$M_1\cong M_2^-$ where $M_2^-$ is obtained from $M_2$ by changing the parity of the highest $\ell$-weight vector.

In general, $\ell$-weight vectors do not need to be eigenvectors of $t_{ii}(u)$. However, from the Gauss decomposition one can deduce that $v$ is a highest $\ell$-weight vector of $\ell$-weight $\bs \zeta $ if and only if $v$ is of parity $p(\bs\zeta)$ and
\beq\label{eq:t-singular}
t_{ij}(u)v=0,\quad t_{kk}(u)v=\zeta_k(u)v,\quad 1\lle i<j\lle m+n,\ k\in \bar I.
\eeq
Note similar formulas do not hold for a lowest $\ell$-weight vector $v$ of $\ell$-weight $\bm \zeta$. 

A $\YglMN$-module $M$ is called \emph{thin} if $M$ has a basis consisting of $\ell$-weight vectors with distinct $\ell$-weights.

Let $v$ and $v'$ be highest $\ell$-weight vectors of $\ell$-weights $\bs \zeta$ and $\bs \vartheta$, respectively. Then, by \eqref{eq:t-singular} and \eqref{eq Hopf}, 
we have
\[
t_{ij}(u)(v\otimes v')=0,\quad t_{kk}(u)(v\otimes v')=\zeta_k(u)\vartheta_k(u)(v\otimes v'), \quad 1\lle i<j\lle m+n,\ k\in \bar I.
\]
Hence $v\otimes v'$ is a highest $\ell$-weight vector of $\ell$-weight $\bs \zeta \bs \vartheta$. In particular, we have
\beq\label{eq:char tensor}
e_{i}(u)(v\otimes v')=0,\quad d_{j}(u)(v\otimes v')=\zeta_j(u)\vartheta_j(u)(v\otimes v'), \quad i\in I,\ j\in \bar I.
\eeq
This formula will be used to obtain information about $\Delta(d_j(u))$.

Every finite-dimensional irreducible $\YglMN$-module is a highest $\ell$-weight module. Let $\bm\zeta \in \mathfrak B$ be an $\ell$-weight. There exists a unique irreducible highest $\ell$-weight $\YglMN$-module of highest $\ell$-weight $\bm\zeta$. We denote it by $L(\bm \zeta)$. The criterion for $L(\bm \zeta)$ to be finite-dimensional is as follows.

\begin{theorem}[\cite{Z96}]The irreducible $\YglMN$-module $L(\bm \zeta)$ is finite-dimensional if and only if there exist monic polynomials $g_i(u)$, $i\in \bar I$, such that
\[
\frac{\zeta_i(u)}{\zeta_{i+1}(u)}=\frac{g_i(u+s_i)}{g_i(u)},\quad \frac{\zeta_m(u)}{\zeta_{m+1}(u)}=\frac{g_m(u)}{g_{m+n}(u)}, \quad i\in I,\ i\ne m,	
\]
and $\deg g_m=\deg g_{m+n}$. 
\end{theorem}

Finite-dimensional irreducible $\YglMN$-modules stay irreducible under restriction to $\YslMN$. Every irreducible finite-dimensional $\YslMN$-module is a restriction of an irreducible finite-dimensional $\YglMN$-module. The restrictions of two 
finite-dimensional irreducible $\YglMN$-modules  are isomorphic $\YslMN$-modules if and only if one of these modules is obtained from the other by a twist by the automorphism $\Gamma_{\vartheta}$ for $\vartheta(u)\in \mathcal B$.

We finish this section by proving the following technical proposition. Define the length function $\ell:{\bf Q}_{\gge 0}\to \Z_{\gge 0}$ by $\ell(\sum_{i\in I}n_i\alpha_i)=\sum_{i\in I}n_i$.
\begin{proposition}\label{prop:cop}
For $i\in I$, $j\in \bar I$, $k\in \Z_{>0}$, we have
\begin{align}
&\Delta({d_j^{(k)}}) -\sum_{l=0}^{k}d_{j}^{(l)}\otimes d_j^{(k-l)}\in \sum_{\ell(\alpha)>0} (\YglMN)_{\alpha}\otimes (\YglMN)_{-\alpha},\label{eq copro d}\\
&\Delta(e_i^{(k)})-1\otimes e_i^{(k)}\in \sum_{\ell(\alpha)>0} (\YglMN)_{\alpha}\otimes (\YglMN)_{\alpha_i-\alpha},\label{eq copro e} \\
&\Delta(f_i^{(k)})-f_i^{(k)}\otimes 1\in \sum_{\ell(\alpha)>0} (\YglMN)_{\alpha-\alpha_i}\otimes (\YglMN)_{-\alpha} \label{eq copro f}.
\end{align}
\end{proposition}
\begin{proof}
We simply write $\mathrm Y_{\alpha}$ for $(\YglMN)_{\alpha}$. Let $\mathrm N_i$ be the subalgebra of $\YglMN$ generated by $e_{j}^{(r)}$ for $r\in \Z_{>0}$, $j\in I\setminus \{i\}$. Let $A_i$ be the unital subalgebra of $\YglMN$ generated by $\phi_i^{(r)}$, $r\in\Z_{>0}$. Let $$h_i^{(2)}=d_i^{(2)}-\frac{1}{2}(d_i^{(1)})^2-\frac{1}{2}d_i^{(1)},$$
then by \eqref{eq com de}, we have $[h_i^{(2)},e_i^{(s)}]=c_i e_i^{(s+1)}$ for some $c_i\in\C^\times$. Note that $d_i^{(2)}=t_{ii}^{(2)}-\sum_{j<i}t_{ij}^{(1)}t_{ji}^{(1)}$ and $d_i^{(1)}=t_{ii}^{(1)}$. Direct computation implies that
\begin{align*}
\Delta(h_i^{(2)})\in h_i^{(2)}\otimes 1+1\otimes h_i^{(2)}+&\ \C^\times e_{i-1}^{(1)}\otimes f_{i-1}^{(1)}\\+&\ \C^\times e_{i}^{(1)}\otimes f_{i}^{(1)} + \sum_{\ell(\alpha)>1}(\mathrm N_i)_\alpha\otimes \mathrm Y_{-\alpha}+\sum_{\ell(\alpha-\alpha_i)>0}(\mathrm N_i)_\alpha\otimes \mathrm Y_{-\alpha}.
\end{align*}
Note that $\Delta(e_i^{(1)})=1\otimes e_i^{(1)}+e_i^{(1)}\otimes 1$. Using $[h_i^{(2)},e_i^{(s)}]=c_i e_i^{(s+1)}$ and $[e_{j}^{(1)}\otimes f_{j}^{(1)},1\otimes e_{i}^{(k)}]=0$ for $j\ne i$ and $k\in\Z_{>0}$, one shows inductively that
\[
\Delta(e_i^{(k)})-1\otimes e_i^{(k)}\in \sum_{s=1}^k e_i^{(s)}\otimes A_i+\sum_{\ell(\alpha)>1}(\mathrm N_i)_\alpha\otimes \mathrm Y_{\alpha_i-\alpha}+\sum_{\ell(\alpha-\alpha_i)>0}(\mathrm N_i)_\alpha\otimes \mathrm Y_{\alpha_i-\alpha}.
\]
In particular, we obtain \eqref{eq copro e}. Similarly, one shows \eqref{eq copro f}.

We then show \eqref{eq copro d}. Since $\phi_i^{(k)}=-(\epsilon_{i+1},\epsilon_{i+1})[e_{i}^{(k)},f_i^{(1)}]$ and $f_i^{(1)}$ supercommutes with $\mathrm N_i$, we have 
$$
\Delta(\phi_i^{(k)})\in A_i\otimes A_i+\sum_{\ell(\alpha)>0}\mathrm Y_\alpha\otimes \mathrm Y_{-\alpha}.
$$ 
Note that $d_1(u)=t_{11}(u)$, we have $\Delta(d_1(u))=d_1(u)\otimes d_1(u)+\sum_{\ell(\alpha)>0}\mathrm Y_\alpha\otimes\mathrm Y_{-\alpha}[[u^{-1}]]$. Hence it suffices to show that
\[
\Delta(\phi_i(u))\in \phi_i(u)\otimes \phi_i(u)+\sum_{\ell(\alpha)>0}\mathrm Y_\alpha\otimes\mathrm Y_{-\alpha}[[u^{-1}]].
\]
Let $\Delta_i(\phi_i^{(k)})\in A_i\otimes A_i$ be such that $\Delta(\phi_i^{(k)})-\Delta_i(\phi_i^{(k)})\in \sum_{\ell(\alpha)>0}\mathrm Y_\alpha\otimes\mathrm Y_{-\alpha}$. Clearly, $A_i=\C[\phi_i^{(k)}]_{k>0}$ is a free polynomial algebra, so is $A_i\otimes A_i$. The elements of a free polynomial algebra are separated by their characters. Therefore, any $x\in A_i\otimes A_i$ is determined by the data $(\chi_1\otimes \chi_2)(x)$ where $\chi_1,\chi_2$ run over all algebra homomorphisms $A_i\to \C$.  Note that $\sum_{\ell(\alpha)>0}\mathrm Y_\alpha\otimes\mathrm Y_{-\alpha}$ annihilates tensor products of highest $\ell$-weight vectors. Therefore, we have
\[
(\chi_1\otimes \chi_2)(\Delta_i(\phi_i^{(k)}))=\sum_{s=0}^k \chi_1(\phi_i^{(s)})\chi_2(\phi_i^{(k-s)})=(\chi_1\otimes \chi_2)\Big(\sum_{s=0}^k \phi_i^{(s)}\otimes \phi_i^{(k-s)}\Big),
\]
where the first equality follows from \eqref{eq:char tensor}. Therefore $\Delta_i(\phi_i^{(k)})=\sum_{s=0}^k \phi_i^{(s)}\otimes \phi_i^{(k-s)}$, completing the proof of \eqref{eq copro d}.
\end{proof}
For quantum affine superalgebra $\mathrm{U}_q(\widehat{\gl}_{m|n})$, the proposition is contained in \cite[Proposition 3.6]{Zhh16}.
A similar (but not the same) statement for Yangians in the even case is proved in \cite[Proposition 2.8]{CP91}.

\subsection{Evaluation maps}
The universal enveloping superalgebra $\mathrm U(\glMN)$ is a Hopf subalgebra of $\YglMN$ via the embedding $e_{ij}\mapsto s_it_{ij}^{(1)}$. The left inverse of this embedding is the \emph{evaluation homomorphism} $\pi_{m|n}: \YglMN\to \UglMN$ given by
\beq\label{eq:evaluation-map}
\pi_{m|n}: t_{ij}(u)\mapsto \delta_{ij}+s_ie_{ij}u^{-1}.
\eeq

The evaluation homomorphism is an algebra homomorphism but not a Hopf algebra homomorphism.
For any $\glMN$-module $M$, it is naturally a $\YglMN$-module obtained by pulling back $M$ through the evaluation homomorphism $\pi_{m|n}$. We denote the corresponding $\YglMN$-module by the same letter $M$ and call it an \emph{evaluation module}.

Following \cite{Naz04}, define the {\it modified evaluation map} of $\YglMN$ by
\[
\pi^\iota_{m|n}:t_{ij}(u)\mapsto \delta_{ij}+(-1)^{(|i|+1)(|j|+1)}e_{ji}u^{-1}.
\]
Given a $\glMN$-module $M$, we call the $\YglMN$-module obtained by 
pulling back  through the modified evaluation map $\pi^\iota_{m|n}$ a {\it modified evaluation module} and denote it by
$\mathbb M$.

The evaluation map and modified one are related as follows. Let $\iota:\YglMN \to \YglMN^{\rm op}$ be the isomorphism of Hopf superalgebras defined by
\beq\label{eq iota}
\iota: t_{ij}(u)\mapsto (-1)^{|i||j|+|i|}t_{ji}(-u),
\eeq
where $\YglMN^{\rm op}$ is the Hopf superalgebra with opposite coproduct \eqref{eq coop}. For a $\YglMN$-module $M$, denote by $M^\iota$ the pull back of $M$ through $\iota$.

Clearly, one has $\pi_{m|n}^\iota=\pi_{m|n}\circ \iota$. 
Therefore, the modified evaluation module can be thought of as the pull back of an evaluation module through the isomorphism $\iota$, namely for a $\glMN$-module $M$, $\mathbb M=M^\iota$. Note that, more generally, for $z\in\C$, we have $\mathbb M_{-z}= (M_z)^\iota$.

Define also the {\it second modified evaluation map} $\pi^\vee_{m|n}:\YglMN\to\UglMN$ 
by
\[
\pi^\vee_{m|n}=\varsigma _{n|m}\circ \pi_{n|m} \circ \hat\varsigma_{m|n}, 
\]
where $\varsigma_{n|m}$, $\pi_{n|m}$, and $\hat\varsigma_{m|n}$ are defined in \eqref{eq change mn}, \eqref{eq:evaluation-map}, and \eqref{eq inter mn yangian}, respectively. The second modified evaluation map will be used in Section \ref{sec div}.

Given a $\glMN$-module $M$, we call the $\YglMN$-module obtained by 
pulling back  through the second modified evaluation map $\pi^\vee_{m|n}$ a {\it second modified evaluation module} and denote it by
$\mathsf M$.

Note that if $v\in M$ is a $\glMN$ singular vector of a given weight, then in the evaluation $\YglMN$-module $M$ and second modified evaluation module $\mathsf M$, $v$ is a highest $\ell$-weight vector, while in the modified evaluation module $\mathbb M$, $v$ is a lowest $\ell$-weight vector.

\subsection{Category $\mathcal C$ and $q$-character map}\label{sec:q-char}
Let $\mathcal C$ be the category of finite-dimensional $\YglMN$-modules. The category $\mathcal C$ is abelian and monoidal. 

Let $M\in \mathcal C$ be a finite-dimensional $\YglMN$-module and $\bm\zeta\in \mathfrak B$ an $\ell$-weight. Let
\[
\zeta_i(u)=1+\sum_{j=1}^\infty \zeta_i^{(j)}u^{-j},\qquad \zeta_i^{(j)}\in \C.
\]
Denote by $M_{\bm \zeta}$ the \emph{generalized $\ell$-weight space} corresponding to the $\ell$-weight $\bm\zeta$,
\[
M_{\bm\zeta}:=\{v\in M~|~  (d_i^{(j)}-\zeta_i^{(j)})^{\dim M} v=0 \text{ for all }i\in \bar I, \ j\in \Z_{>0}, \text{ and } |v|=p(\bs\zeta) \}.
\]

For a finite-dimensional $\YglMN$-module $M$, define the $q$-{\it character} (or \emph{Yangian character}) of $M$ by the element 
\[
\chi(M):=\sum_{\bm\zeta\in \mathfrak B}\dim(M_{\bm \zeta})[\bm\zeta]\in \Z[\mathfrak B].
\]
This definition is a straightforward generalization of the even case, see \cite{Kni:1995}. It is also called the {\it Gelfand-Tsetlin character}, see \cite[Definition 8.5.7]{Mol07}, since the commutative subalgebra generated by coefficients of $d_i(u)$ for all $i\in \bar I$ is usually called the {\it Gelfand-Tsetlin subalgebra} of $\YglMN$.

Let $\mathscr Rep(\mathcal C)$ be the Grothendieck ring of $\mathcal C$, then $\chi$ induces a $\Z$-linear map from $\mathscr Rep(\mathcal C)$ to $\Z[\mathfrak B]$.

Define the map $\varpi:\mathfrak B\to \h^*,\ \bm\zeta\mapsto \varpi(\bm\zeta)$ by $\varpi(\bm\zeta)(e_{ii})=s_i
\zeta_i^{(1)}.$
\begin{lemma}\label{lem chi morphism}
The map $\chi:\mathscr Rep(\mathcal C)\to \Z[\mathfrak B]$ is an injective ring homomorphism. 
\end{lemma}
\begin{proof}
The fact that $\chi:\mathscr{R}ep(\mathcal C)\to \Z[\mathfrak B]$ is an ring homomorphism follows from Proposition \ref{prop:cop}, see e.g. \cite[Theorem 2]{Kni:1995}. Since $L(\bm \zeta)$ is of highest $\ell$-weight, by Theorem \ref{thm PBW}, $\chi(L(\bm\zeta))$ is equal to $[\bm\zeta]$ plus $\ell$-weights of form $[\bm\xi]$ such that $\varpi(\bm\xi)$ is strictly smaller than $\varpi(\bm\zeta)$ with respect to the partial ordering on $\h^*$. Therefore $[\bm\zeta]$ is the leading term in $\chi(L(\bm\zeta))$. Now the injectivity of $\chi$  is clear. 
\end{proof}
In particular, we obtain the following.
\begin{corollary}\label{cor gro-ring-comm}
The Grothendieck ring $\mathscr Rep(\mathcal C)$ is commutative.
\end{corollary}

\section{Skew representations and Jacobi-Trudi identity}\label{sec skew}
\subsection{Skew representations}
Consider the embedding of $\gl_{m'|n'}$ into $\gl_{m'+m|n'+n}$ sending $e_{ij}$ to $e_{ij}$ for $i,j=1,2,\dots,m'+n'$. Here $\gl_{m'|n'}$ has the standard parity and $\gl_{m'+m|n'+n}$ has parity \eqref{eq parity}.

Let $\la$ and $\mu$ be an $(m'+m|n'+n)$-hook partition and an $(m'|n')$-hook partition, respectively. Suppose further that $\la_i\gge \mu_i$ for all $i\in \Z_{>0}$. Consider the skew Young diagram $\la/\mu$.

Let $\mu^\natural$ be the $\gl_{m'|n'}$-weight corresponding to $\mu$, see  \eqref{eq natural}, and let $\la^\circ$ be the $\gl_{m'+m|n'+n}$-weight corresponding to $\la$, see \eqref{eq circ}. We have the finite-dimensional irreducible $\gl_{m'+m|n'+n}$-module $L(\la^\circ)$.  Consider $L(\la^\circ)$ as a $\gl_{m'|n'}$-module. 

Define $L(\la/\mu)$ to be the subspace of $L(\la^\circ)$ by
\[
L(\la/\mu):=\{v\in L(\la^\circ)~|~e_{ii}v=\mu^{\natural}(e_{ii})v,\ e_{jk}v=0,\, i=1,2,\dots,m'+n',\ 1\lle j<k\lle m'+n'\}.
\]
The subspace $L(\la/\mu)$ has a natural $\mathrm{U}(\gl_{m'+m|n'+n})^{\gl_{m'|n'}}$-module structure.

Let $\varphi_{m'|n'}:\YglMN\to \mathrm Y(\gl_{m'+m|n'+n})$ be the embedding given by
\[
\varphi_{m'|n'}: t_{ij}(u)\mapsto t_{m'+n'+i,m'+n'+j}(u).
\]
Recall $\eta_{m|n}$ (and $\eta_{m'+m|n'+n}$) from \eqref{eq iso antipode}. Let $\psi_{m'|n'}:\YglMN\to \mathrm Y(\gl_{m'+m|n'+n})$ be the injective homomorphism given by
\[
\psi_{m'|n'}:=\eta_{m'+m|n'+n}\circ \varphi_{m'|n'}\circ \eta_{m|n}.
\]
The following lemma can be found in \cite[Proof of Lemma 4.2]{Peng}.

\begin{lemma}[\cite{Gow07,Peng}]\label{eq psi map}
We have
\[
\psi_{m'|n'}(d_i(u))=d_{m'+n'+i}(u),\quad \psi_{m'|n'}(e_i(u))=e_{m'+n'+i}(u),\quad \psi_{m'|n'}(f_i(u))=f_{m'+n'+i}(u).
\]
\end{lemma}
Regard $\mathrm{Y}(\gl_{m'|n'})$ as the subalgebra of $\mathrm Y(\gl_{m'+m|n'+n})$ via the natural embedding $t_{ij}(u)\mapsto t_{ij}(u)$ for $i,j=1,\dots,m'+n'$. We have the following lemma from \cite[Lemma 4.3]{Peng}.
\begin{lemma}[\cite{Peng}]\label{lem commute yangian}
The subalgebra $\mathrm{Y}(\gl_{m'|n'})$ of $\mathrm Y(\gl_{m'+m|n'+n})$ supercommutes with the image of $\YglMN$ under the map $\psi_{m'|n'}$.
\end{lemma}
Recall the evaluation map $\pi$, see \eqref{eq:evaluation-map}, the following is straightforward from Lemma \ref{lem commute yangian}. 

\begin{corollary}\label{prop centralizer}
The image of the homomorphism $$\pi_{m'+m|n'+n}\circ \psi_{m'|n'}:\YglMN\to \mathrm{U}(\gl_{m'+m|n'+n})$$ supercommutes with the subalgebra $\mathrm{U}(\gl_{m'|n'})$ in $\mathrm{U}(\gl_{m'+m|n'+n})$.
\end{corollary}

Corollary \ref{prop centralizer} implies that the subspace $L(\la/\mu)$ is invariant under the action of the image of $\pi_{m'+m|n'+n}\circ \psi_{m'|n'}\circ \tau_{m'-n'}$. Therefore, $L(\la/\mu)$ is a $\YglMN$-module. We call $L(\la/\mu)$ a {\it skew representation}, cf. \cite{Che87b,Che89}.

We study the skew representations in the rest of this section. We will show that all skew representations of $\YglMN$ are irreducible, see Theorem \ref{thm skew} below.

\subsection{$q$-characters of skew representations}
In this section we compute the $q$-character of the $\YglMN$-module $L(\la/\mu)$. 

Let $\kappa_i=i-1$ if $i=1,\dots,m$ and $\kappa_i=2m-i$ if $i=m+1,\dots,m+n$. For each $a\in \C$ and $i\in \bar I$, let
\[
\mathscr X_{i,a}= \left(1,\dots, \big(1+(u+a+\kappa_i)^{-1}\big)^{s_i}, \dots,1\right)^{|i|} \in \mathfrak B.
\]
Here the only component not equal to 1 is at the $i$-th position.

Recall that $\mathcal T(i,j)$ and $c(i,j)=j-i$ denote the number in the box representing the pair $(i,j)\in \la/\mu$ and the content of the pair $(i,j)$ for a semi-standard Young tableau $\mathcal T$ of shape $\la/\mu$, respectively.

It is known from \cite{CPT15} that the dimension of $L(\la/\mu)$ is equal to the number of semi-standard Young tableaux of shape $\la/\mu$. The following theorem is a refinement of this statement which is a super analog of \cite[Lemma 2.1]{NT98}.
\begin{theorem}\label{thm character}
The $q$-character of the $\YglMN$-module $L(\la/\mu)$ is given by
\[
\mathscr{K}_{\la/\mu}(u):=\chi(L(\la/\mu))=\sum_{\mathcal T}\prod_{(i,j)\in\la/\mu}\mathscr X_{\mathcal T(i,j),c(i,j)},
\]
summed over all semi-standard Young tableaux $\mathcal T$ of shape $\la/\mu$. In particular, $L(\la/\mu)$ is thin.
\end{theorem}
Before proving the theorem, we recall the following proposition from \cite[Theorems 1 and 2]{Gow05} and \cite[Theorem 2.43]{Tsy}. 

Similar to $\kappa_i$, define $\kappa_i'$, $i=1,\dots,m'+n'$, with $m$ and $n$ replaced by $m'$ and $n'$, respectively. Set $\kappa_{m'+n'+j}'=m'-n'+\kappa_j$, $j\in \bar I$. Let $s'_i=(-1)^{|i|}$, $i=1,2,\dots,m'+n'+m+n$, such that $|i|$ are chosen as in \eqref{eq parity}.
\begin{proposition}[\cite{Gow05,Tsy}]\label{prop center}
The coefficients of the series $\prod_{i\in \bar I}(d_i(u-\kappa_i))^{s_i}$ are central in $\YglMN$. The coefficients of the series $\prod_{i=1}^{m'+n'+m+n}(d_i(u-\kappa'_i))^{s'_i}$ are central in $\mathrm{Y}(\gl_{m'+m|n'+n})$.
\end{proposition}
\begin{lemma}\label{lem center action}
Let $\la$ be a Young diagram. Then the operator $\prod_{i\in \bar I}(d_i(u-\kappa_i))^{s_i}$ acts on the evaluation $\YglMN$-module $L(\la^\natural)$ by the scalar operator
\[
\prod_{i\in \bar I}(d_i(u-\kappa_i))^{s_i}\Big|_{L(\la^\natural)}=\prod_{i\in \bar I}\Big(1+\frac{(\la_i^\natural,\epsilon_i)}{u-\kappa_i}\Big)^{s_i}=\prod_{(i,j)\in \la}\frac{u+c(i,j)+1}{u+c(i,j)}.
\]
Similarly, the operator $\prod_{i=1}^{m'+n'+m+n}(d_i(u-\kappa'_i))^{s'_i}$ acts on the evaluation $\mathrm{Y}(\gl_{m'+m|n'+n})$-module $L(\la^\circ)$ by the scalar operator
\beq\label{eq:ber-action-2}
\prod_{i=1}^{m'+n'+m+n}(d_i(u-\kappa'_i))^{s'_i}\Big|_{L(\la^\circ)}=\prod_{(i,j)\in \la}\frac{u+c(i,j)+1}{u+c(i,j)}.
\eeq
\end{lemma}
\begin{proof}
The statement follows from Proposition \ref{prop center} and direct computations on highest $\ell$-weight vector.
\end{proof}

\begin{proof}[{Proof of Theorem} \ref{thm character}]
The proof of the first statement is similar to \cite[Lemma 2.1]{NT98} and \cite[Lemma 4.7]{FM02}. We sketch the proof following the exposition of \cite[Corollary 8.5.8]{Mol07}. 

We first show it for the case when $m'+n'=0$. Then $L(\la/\mu)$ is indeed the evaluation module $L(\la^\natural)$. For a semi-standard Young tableau $\mc T$, denote by $\mc T_k$ the sub-tableau consisting of all boxes occupied by integers $1,\dots,k$ for $k=0,1,\dots,m+n$. We have $\varnothing =\mc T_0\subset\mc T_1\subset \mc T_2\subset \cdots\subset \mc T_{m+n}=\mc T$. Moreover, $\mc T$ is uniquely determined by the data $(\mc T_1,\dots,\mc T_{m+n})$. The $\glMN$-module $L(\la^\natural)$ has a basis $\{v_{\mc T}\}$ indexed by semi-standard Young tableaux of shape $\la$ which is called the Gelfand-Tsetlin basis and is obtained from the branching rule of $\glMN$, see \cite{BR83} for the standard parity sequence and \cite{CPT15} for arbitrary parity sequences. Using the properties of Gelfand-Tsetlin basis and Lemma \ref{lem center action}, we obtain that
\beq\label{eq:111}
(d_k(u-\kappa_k))^{s_k}\ v_{\mc T}=\prod_{(i,j)\in \mc T_{k}/\mc T_{k-1}}\frac{u+c(i,j)+1}{u+c(i,j)},
\eeq
completing the proof in the case $m'+n'=0$. In the skew case namely $m'+n'>0$, we use \eqref{eq:ber-action-2} and Lemma \ref{eq psi map}. Then formulas \eqref{eq:111} remain valid as well. Note the shift automorphism $\tau_{m'-n'}$ in the definition of skew representations.

The second statement follows from the fact that different semi-standard Young tableaux $\mathcal T$ of the same shape correspond to different $\ell$-weights. Indeed, the data $(\mc T{(i,j)}, c(i,j))$ determine the semi-standard Young tableau uniquely, since the content (the second component of the pair) tells us which diagonal it belongs to, and on the same diagonal the numbers occupying the boxes (the first component of the pair) strictly increase.
\end{proof}

\begin{rem}
Due to Theorem \ref{thm character}, the $q$-character of $L(\la/\mu)$ relies only on the shape $\la/\mu$ and not on $m',n'$. Thus we have the module $L(\la/\mu)$ for arbitrary skew Young diagram $\la/\mu$ and its $q$-character is given by Theorem \ref{thm character}. Indeed, one can enlarge $m'$ such that both $\la$ and $\mu$ are hook partitions (of different indices). Moreover, we can also let $n'=0$, then the parity \eqref{eq parity} is a standard parity sequence and the Lie superalgebra $\mathfrak{gl}_{m'+m|n'+n}$ is associated to a standard parity sequence. Therefore, one can always set $n'=0$ to simplify the discussion. 
\end{rem}

For the $\YglMN$-module $L(\la/\mu)$, we write $L_z(\la/\mu)$ for $(L(\la/\mu))_z$. It is clear that
\[
\chi(L_z(\la/\mu))=\mathscr K_{\la/\mu}(u-z).
\]

\begin{rem}
Recall that the pair $(i,j)\in \la/\mu$ is represented by the unit box whose south-eastern corner has coordinate $(i,j)\in \Z^2$. We can shift the whole Young diagram so that the numbers in the pair $(i,j)\in \mathbb R^2$ are not necessarily integers, but real numbers. The shape of the diagram remains the same. Only the contents of all boxes are shifted by the same number simultaneously. Hence we may consider the entire diagram is fixed and for any $z\in\C$ we can define the content $c_z$ in a more general way, $c_z(i,j)=j-i-z$. Clearly, the $q$-character $\mathscr K_{\la/\mu}(u-z)$ is written in the same way as $\mathscr K_{\la/\mu}(u)$ 
in Theorem \ref{thm character} by changing $c(i,j)$ to $c_z(i,j)$.

Note the contents of all boxes are uniquely determined by the content of a single box. Therefore, in the following, we shall sometimes specify the content of a box. If no content of a box is specified, then we are using the standard definition of content, namely $c(i,j)=j-i$. We also remark that for different Young diagrams $\la,\tl \la,\mu,\tl \mu$, the diagrams $\la/\mu$ and $\tl \la/\tl \mu$ may have the same shape but with possibly different contents.
\end{rem}

\subsection{Divisibility of $q$-characters}\label{sec div}
In this section, we discuss the divisibility of $q$-characters of skew representations associated to special skew Young diagrams. We expect these observations would be helpful to understand Theorem \ref{thm ratio} and prove Conjecture \ref{conj bae} below, see Section \ref{sec spectra}. 

For a partition $\lambda$, denote by $\la^-$ the skew Young diagram obtained by rotating $\lambda$ by 180 degrees. By convention, we set the content of bottom-right box of $\la^-$ to be zero.

\begin{eg}
Let $\la$ be the partition $(2,1,1)$, then $\la$ and $\la^-$ are given by $\ytableausetup{smalltableaux,centertableaux}
\ytableaushort[*(white)]
{0~,~\none,~\none}$ and 
$\ytableausetup{centertableaux}
\ytableaushort[*(white)]
{\none~,\none~,~0}$, respectively. Here the number $0$ stands for the content of the corresponding box.
\end{eg}
Note that semi-standard Young tableaux of shape $\la'$ are in one-to-one correspondence with semi-standard Young tableaux of shape $\la^-$ given by changing numbers $i$ in boxes of $\la'$ to $m+n-i+1$ in the corresponding boxes of $\la^-$, see e.g. \cite[Lemma 2.7]{Zhh18}. Moreover, $\la^-$ is the reflection of $\la'$ with respect to the line $x+y=0$ (recall our convention on the coordinate plane $\mathbb R^2$ from Section \ref{sec:hook-p}) and this reflection preserves the contents of boxes.

Let $\mathsf L_z(\lambda^\natural)$ be the second modified evaluation module of $L(\la^\natural)$ twisted by $\tau_z$ and set $\mathsf L(\lambda^\natural):=\mathsf L_0(\lambda^\natural)$.

\begin{lemma}[{\cite[Theorem 2.4]{Zhh18}}]\label{lem *eva}
Up to a parity, we have the isomorphism of $\YglMN$-modules, $$\mathsf L_z(\lambda^\natural)\cong  L_{m-n+z}(\la^-).$$
\end{lemma}
\begin{proof}
Clearly, $\mathsf L_z(\lambda^\natural)$ is irreducible and hence it suffices to check that both sides have the same $q$-character (up to a parity). Recall the second modified evaluation map $\pi^\vee_{m|n}=\varsigma _{n|m}\circ \pi_{n|m} \circ \hat\varsigma_{m|n}$. The $\gl_{n|m}$-module obtained by pulling back the $\glMN$-module $L(\la^\natural)$ through $\varsigma_{n|m}$ is isomorphic to the $\gl_{n|m}$-module $L((\la')^\natural)$. Let $\mc T'$ be a semi-standard $\gl_{n|m}$ Young tableau of shape $\la'$. Denote by $\mc T^-$ the semi-standard $\gl_{m|n}$ Young tableau of shape $\la^-$ obtained by the correspondence described above. Introduce sub-tableaux $\mc T'_i$ and $\mc T_i^-$ for each $i=0,1,\dots,m+n$ in the same way as in the proof of Theorem \ref{thm character}. We use tilde to distinguish the $\gl_{n|m}$ notation from $\glMN$ notation. Let $v_{\mc T'}$ be the vector in the $\mathrm{Y}(\gl_{n|m})$-evaluation module $L_z((\la')^\natural)$ corresponding to $\mc T'$. By \eqref{eq:111}, we have
\[
\tl d_k(u)^{\tl s_k}\ v_{\mc T'}=\prod_{(i,j)\in \mc T'_{k}/\mc T'_{k-1}}\frac{u-z+c(i,j)+1+\tl \kappa_k}{u-z+c(i,j)+\tl \kappa_k}.
\]
The $\YglMN$-module $\mathsf L_z(\lambda^\natural)$ is identified with the pullback of the $\mathrm{Y}(\gl_{n|m})$-evaluation module $L_z((\la')^\natural)$ via $\hat\varsigma_{m|n}$. By Lemma \ref{lem iso mn inter}, using $-\tl s_{k}=s_{m+n+1-k}$, we obtain
\[
d_{m+n+1-k}(u)^{s_{m+n+1-k}}\ v_{\mc T'}=\prod_{(i,j)\in \mc T_{m+n+1-k}^-/\mc T_{m+n-k}^-}\frac{u-z+c(i,j)+1+\tl \kappa_k}{u-z+c(i,j)+\tl \kappa_k}.
\]
Here we also used the fact that the refection of $\mc T'_{k}/\mc T'_{k-1}$ with respect to the line $x+y=0$ is exactly $\mc T_{m+n+1-k}^-/\mc T_{m+n-k}^-$. Now the lemma follows from the equality $\kappa_{m+n+1-k}=\tl \kappa_k+m-n$.
\end{proof}

Let $\Xi$ be the partition whose corresponding Young diagram is a rectangle of size $m\times n$. For a partition $\la\in\mathscr P(m|0)$, define $\mathbf{W}(\la)$ to be the skew Young diagram obtained by gluing $\Xi$ and $\la^-$ so that the bottom row of $\la^-$ is next to the bottom row of the rectangular one exactly from left. Similarly, for $\mu\in \mathscr P(0|n)$, define $\mathbf{S}(\mu)$ to be the Young diagram obtained by attaching $\mu$ to the bottom of $\Xi$ such that the first column of $\lambda$ is exactly below the first column of $\Xi$. Moreover, we always assume that the box at left-upper corner of $\Xi$ has content zero.

\begin{eg}
Consider the case $m=4$ and $n=3$. Let $\la=(2,1,1)$ and $\mu=(3,2,2)$. Then the skew Young diagram $\mathbf{W}(\la)$ and the Young diagram $\mathbf{S}(\mu)$ are as follows.
$$
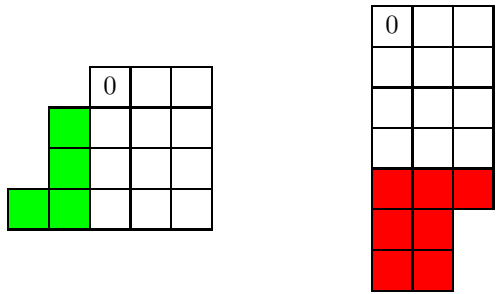
\ytableausetup{nosmalltableaux,centertableaux}
\ytableaushort[*(white)]
{\none\none 0~~,\none{*(green)~}~~~,\none{*(green)~}~~~,{*(green)~}{*(green)~}~~~}\qquad \qquad\qquad
{
\ytableaushort[*(white)]
{0~~,~~~,~~~,~~~,{*(red)~}{*(red)~}{*(red)~},{*(red)~}{*(red)~},{*(red)~}{*(red)~}}
}
$$
Here the number zero means that the contents of the corresponding boxes are zero. We use the green color and red color to indicate the diagrams $\la^-$ and $\mu$ corresponding to partitions $\la$ and $\mu$, respectively.
\end{eg}

Denote by $\chi_m$ the $q$-character map of $\mathrm{Y}(\gl_{m})$. We also have the evaluation modules for $\mathrm{Y}(\gl_{m})$, $\mathrm{Y}(\gl_{0|n})$ defined by setting $n=0$ and $m=0$, respectively. Recall that $\mathscr{K}_{\la/\mu}(u)=\chi(L(\la/\mu))$, we use the superscripts to indicate the underlying algebra, e.g. $\mathscr{K}^{m|n}_{\la/\mu}(u)$, $\mathscr{K}^{m|0}_{\la/\mu}(u)$, and $\mathscr{K}^{0|n}_{\la/\mu}(u)$.

We identify an $\ell$-weight $\bm\zeta_{[m]}=(\zeta_i(u))_{1\lle i\lle m}^{|0|}$ for $\mathrm{Y}(\gl_{m})$ with an $\ell$-weight $\bm\zeta=(\zeta_i(u))_{i\in\bar I}^{|0|}$ for $\YglMN$ via the natural embedding $\mathrm{Y}(\gl_m)\hookrightarrow \YglMN$, where $\zeta_j(u)=1$ for $j=m+1,\dots,m+n$. Similarly, an $\ell$-weight $\bm\vartheta_{[n]}=(\vartheta_i(u))_{1\lle i\lle n}^{p(\bm\vartheta_{[n]})}$ for $\mathrm{Y}(\gl_{0|n})$ is identified with an $\ell$-weight $\bm\vartheta=(\vartheta_i(u))_{i\in\bar I}^{p(\bm\vartheta_{[n]})}$ for $\YglMN$ via the embedding $\psi_{m|0}:\mathrm{Y}(\gl_{0|n})\hookrightarrow \YglMN$ in Lemma \ref{eq psi map}, where $\vartheta_j(u)=1$ for $j=1,\dots,m$. 

\begin{lemma}\label{lem character q-div}
We have the equalities of $q$-characters
\beq\label{eq q-div east}
\mathscr K^{m|n}_{\mathbf W(\la)}(u)=\mathscr K^{m|0}_{\la^-}(u-m)\cdot \mathscr K^{m|n}_{\Xi}(u),
\eeq
\beq\label{eq q-div south}
\mathscr K^{m|n}_{\mathbf S(\mu)}(u)=\mathscr K^{0|n}_{\mu}(u-m)\cdot \mathscr K^{m|n}_{\Xi}(u).
\eeq
\end{lemma}
\begin{proof}
We first show \eqref{eq q-div east}. It is not hard to see that all semi-standard Young tableaux of shape $\mathbf{W}(\la)$ are obtained by independently filling numbers $1,\dots,m+n$ to the rectangle $\Xi$ and numbers $1,\dots,m$ to $\la^-$ so that each part gives a semi-standard Young tableau. In particular, $\mathscr K^{m|n}_{\mathbf W(\la)}(u)$ can be written as the product of $\mathscr K^{m|n}_{\Xi}(u)$ and the summation of monomials in $\mathscr X_{i,a}$ associated to semi-standard Young tableaux corresponding to $\la^-$ filled by numbers $1,\dots,m$. Hence we obtain \eqref{eq q-div east}.

Similarly, all semi-standard Young tableaux of shape $\mathbf{S}(\mu)$ are obtained by independently filling numbers $1,\dots,m+n$ to the rectangle $\Xi$ and numbers $m+1,\dots,m+n$ to $\mu$ so that each part gives a semi-standard Young tableau. The equality \eqref{eq q-div south} is proved in a similar way.
\end{proof}

The lemma does not imply an equality on the representation level as not all $q$-characters are for $\YglMN$.

\begin{corollary}
We have the equality of $q$-characters
\beq\label{eq q-div east 2}
\mathscr K^{m|n}_{\mathbf W(\la)}(u)=\chi_{m}(\mathsf L(\la_{[m|0]}^\natural))\cdot \mathscr K^{m|n}_{\Xi}(u).
\eeq
\end{corollary}
\begin{proof}
The statement follows from Lemma \ref{lem *eva} and Lemma \ref{lem character q-div}.
\end{proof}
\begin{rem}
Equations \eqref{eq q-div east} and \eqref{eq q-div south} are related by Lemma \ref{lem *eva}. Namely, ignoring the contents, we have the equality for skew Young diagrams
\[
\big(\mathbf W_{n|m}(\mu')\big)'=\big(\mathbf S_{m|n}(\mu)\big)^{-},
\]
where $\mathbf S_{m|n}(\mu):=\mathbf S(\mu)$ is defined above while $\mathbf W_{n|m}(\mu')$ is defined in the same way as $\mathbf W_{m|n}(\la):=\mathbf W(\la)$ with $m$ and $n$ interchanged.
\end{rem}
\begin{rem}
Equation \eqref{eq q-div south} implies that the $\gl_{m|n}$-character of $L(\mathbf S(\mu)^\natural)$ is equal to the $\gl_{m|n}$-character of $L(\Xi^\natural)$ multiplied by the $\gl_{n}$-character of $L((\mu^\prime)_{[n|0]}^\natural)$. This fact can be understood observing that the $\gl_{m|n}$-module $L(\mathbf S(\mu)^\natural)$ is an irreducible Kac module. Indeed, let $\g_{+1}$ and $\g_{-1}$ be the odd subspaces of $\glMN$ spanned by $e_{i,m+j}$ and $e_{m+j,i}$ for all $i=1,\dots,m$ and $j=1,\dots,n$, respectively. Note the $\gl_{m}$-module $L(\Xi^\natural_{[m|0]})$ is one-dimensional. Extend the $\gl_m\oplus\gl_n$-module $L(\Xi^\natural_{[m|0]}) \otimes L((\mu^\prime)_{[n|0]}^\natural)$ to the $\gl_m\oplus\gl_n\oplus \g_{+1}$-module by putting $\g_{+1}\big(L(\Xi^\natural_{[m|0]}) \otimes L((\mu^\prime)_{[n|0]}^\natural)\big)=0$ , then we have the isomorphism of vector spaces by PBW theorem for $\glMN$,
\[
L(\mathbf S(\mu)^\natural)=\mathrm{Ind}^{\glMN}_{\gl_m\oplus\gl_n\oplus \g_{+1}}\big(L(\Xi^\natural_{[m|0]}) \otimes L((\mu^\prime)_{[n|0]}^\natural)\big)\cong L((\mu^\prime)_{[n|0]}^\natural)\otimes \wedge^\bullet[\g_{-1}],
\]
where $\wedge^\bullet[\g_{-1}]$ denotes the Grassmann algebra with $mn$ variables and hence has dimension $2^{mn}$. The part $\mathscr K^{0|n}_{\mu}(u-m)$ corresponds to $L((\mu^\prime)_{[n|0]}^\natural)$ while $\mathscr K^{m|n}_{\Xi}(u)$ corresponds to $\wedge^\bullet[\g_{-1}]$.

Equations \eqref{eq q-div east} and \eqref{eq q-div east 2} can be interpreted similarly as well.
\end{rem}

\subsection{Jacobi-Trudi type identity}\label{Sec JT}
Set $\mathscr S_{k}(u)=\mathscr K_{\la/\mu}(u)$ if $\la=(k)$ and $\mu=(0)$, and $\mathscr A_{k}(u)=\mathscr K_{\la/\mu}(u)$ if $\la=(1^k)$ and $\mu=(0)$, where $k\in \Z_{>0}$.

We have the Jacobi-Trudi type identity for $q$-characters of skew representations.
\begin{theorem}\label{thm Jacobi-Trudi}
We have
\begin{align*}
\mathscr K_{\la/\mu}(u)\,&=\det_{1\lle i,j\lle \la_1'}\mathscr S_{\la_i-\mu_j-i+j}(u+\mu_j-j+1)\\ &= \det_{1\lle i,j\lle \la_1}\mathscr A_{\la_i'-\mu'_j-i+j}(u-\mu'_j+j-1).
\end{align*}
Here we use the convention that $\mathscr S_k(u)=\mathscr A_k(u)=0$ for $k<0$ and $\mathscr S_0(u)=\mathscr A_0(u)=1$.
\end{theorem}
\begin{proof}
We give a proof of the first equality using the Lindstr\"{o}m-Gessel-Viennot lemma, the second equality is proved similarly. We refer the reader to \cite[Chapter 4.5]{Sag01} for a detailed description of the method.  Here we only give the necessary modifications to our situation.  We remark that our adjustment is essentially the same as that of \cite[Theorem 3.1]{Mol97}.

Consider the lattice plane $\Z^{2}$ (like the usual $x$-$y$ coordinate plane and it should not be confused with the one defining Young diagrams) and lattice paths from one point to another. The paths start at the line $y=1$, consist of steps from one point to another of unit length northward or eastward or of length $\sqrt{2}$ northeastward and end at the line $y=m+n+1$. More precisely, the step starting from the point $(i,j)$ can end at $(i+1,j)$ or $(i,j+1)$ if $1\lle j\lle m$ and at $(i,j+1)$ or $(i+1,j+1)$ if $m< j\lle m+n$. We call an eastward or northeastward step a {\it contributed step}. For a contributed step $\mathfrak s$, denote by $\mathfrak s_x$ and $\mathfrak s_y$ the $x$-coordinate and $y$-coordinate of the starting point of $\mathfrak s$, respectively. For a lattice path $\mathfrak p$, define a monomial $\mathscr X^{\mathfrak p}$ in $\mathscr X_{i,a}$ by
\[
\mathscr X^{\mathfrak p} = \prod_{\text{contributed }\mathfrak s\in \mathfrak p}\mathscr X_{\mathfrak s_y,\mathfrak s_x}.
\]
Suppose the path $\mathfrak p$ is from the point $(i_0,j_0)$ to the point $(i_1,j_1)$, then it is clear that $\mathscr X^{\mathfrak p}$ is a term in $\mathscr S_{i_1-i_0}(u+i_0)$ as in Theorem \ref{thm character}. Moreover, we have
\[
\mathscr S_{i_1-i_0}(u+i_0)=\sum_{\mathfrak p}\mathscr X^{\mathfrak p}
\]
summed over all lattice paths $\mathfrak p$ from $(i_0,j_0)$ to $(i_1,j_1)$.

Set $l=\la_1'$. Let $\sigma$ be an element of the symmetric group $\mathfrak S_{l}$ permuting numbers $1,\dots,l$. Let $\mathfrak p^{\sigma}_{i}$ be a lattice path from the point $(\mu_i-i+1,1)$ to the point $(\la_{\sigma(i)}-\sigma(i)+1,m+n+1)$ for $i=1,\dots,l$. Consider an $l$-tuple of paths $\bm p(\sigma)=(\mathfrak p^{\sigma}_{1},\dots,\mathfrak p^{\sigma}_{l})$. Define the monomial $\mathscr X^{\bm p(\sigma)}$ associated to $\bm p(\sigma)$ by
\[
\mathscr X^{\bm p(\sigma)}=\prod_{i=1}^{l}\mathscr X^{\mathfrak p_i^{\sigma}}
\]
if all $\mathfrak p^{\sigma}_{i}$ exist, otherwise set $\mathscr X^{\bm p(\sigma)}=0$. Then one has
\beq\label{eq path expansion}
\det_{1\lle i,j\lle l}\mathscr S_{\la_i-\mu_j-i+j}(u+\mu_j-j+1)=\sum (-1)^{\mathrm{sign}(\sigma)}\mathscr X^{\bm p(\sigma)},
\eeq
where the summation is over all possible $\sigma\in \mathfrak S_{l}$ and all possible $l$-tuples $\bm p(\sigma)$.

We call a tuple $\bm p(\sigma)$ an \emph{intersecting tuple} if $\mathfrak p_i^\sigma$ intersects $\mathfrak p_j^\sigma$ for some $i\ne j$. All monomials in \eqref{eq path expansion} corresponding to intersecting tuples are cancelled out, see the proof of \cite[Theorem 4.5.1]{Sag01}, below (4.16) therein. A tuple of the form $\bm p(\sigma)$ is non-intersecting only if $\sigma = \mathrm{id}$. Moreover, there exists a bijection between non-intersecting tuples with semi-standard Young tableaux of shape $\la/\mu$. One direction of this bijection is described as follows, see an explicit example given below. Let $\bm p(\mathrm{id})=(\mathfrak p_1^\mathrm{id},\dots,\mathfrak p_l^\mathrm{id})$ be a non-intersecting tuple, then filling the $i$-th row of the skew Young diagram $\la/\mu$ with the numbers $\mathfrak s_y$ for all contributed steps $\mathfrak s$ of $\mathfrak p_i^\mathrm{id}$ in non-decreasing order, for all $i=1,\dots,l$, gives a semi-standard Young tableau of shape $\la/\mu$. Note that $\mathfrak s_x$ corresponds to the content of the box in the $i$-th row filled with the number $\mathfrak s_y$. The theorem now follows from this bijection and Theorem \ref{thm character}.
\end{proof}

We exhibit an example explaining the correspondence between non-intersecting tuples and semi-standard Young tableaux.

\begin{eg} Let $\la=(4,3,2)$, $\mu=(1,1,0)$, $m=n=2$. Consider the following semi-standard Young tableau.

$$\mathcal T=\ \  \ytableausetup{nosmalltableaux,centertableaux}
\ytableaushort[*(white)]
{\none 122,\none34,23}$$

\medskip

In particular, the  content of the box filled with number $1$ is $1$.
Then it corresponds to the following $3$-tuple of lattice paths.
$$
\begin{tikzpicture}[scale=0.75]
\tikzset{edge/.style = {->,> = latex'}}
\draw (1,1) grid (7,5);
\foreach \i in {1,...,7}
    \foreach \j in {1,...,5}
        \fill (\i,\j) circle (1pt);
\draw[edge,ultra thick,red] (4,1) -- (5,1) node[pos=.5,above] {${\color{red} 1}$};
\draw[edge,ultra thick,red] (5,1) -- (5,2);
\draw[edge,ultra thick,red] (5,2) -- (6,2) node[pos=.5,above] {${\color{red} 2}$};
\draw[edge,ultra thick,red] (6,2) -- (7,2) node[pos=.5,above] {${\color{red} 2}$};
\draw[edge,ultra thick,red] (7,2) -- (7,3);
\draw[edge,ultra thick,red] (7,3) -- (7,4);
\draw[edge,ultra thick,red] (7,4) -- (7,5);

\draw[edge,ultra thick,blue] (3,1) -- (3,2);
\draw[edge,ultra thick,blue] (3,2) -- (3,3);
\draw[edge,ultra thick,blue] (3,3) -- (4,4) node[pos=.4,above] {${\color{blue} 3}$};
\draw[edge,ultra thick,blue] (4,4) -- (5,5) node[pos=.4,above] {${\color{blue} 4}$};

\draw[edge,ultra thick,green] (1,1) -- (1,2);
\draw[edge,ultra thick,green] (1,2) -- (2,2) node[pos=.5,above] {${\color{green} 2}$};
\draw[edge,ultra thick,green] (2,2) -- (2,3);
\draw[edge,ultra thick,green] (2,3) -- (3,4) node[pos=.4,above] {${\color{green} 3}$};
\draw[edge,ultra thick,green] (3,4) -- (3,5);
\draw[edge] (7,1) -- (8,1) node[pos=0.9,below] {$x$};
\draw[edge] (1,5) -- (1,6) node[pos=0.9,left] {$y$};

\fill[black] (1,1) circle (3pt) node[left] {1} 
             (1,2) circle (1pt) node[left] {2} 
             (1,3) circle (1pt) node[left] {3}
             (1,4) circle (1pt) node[left] {4}
             (1,5) circle (1pt) node[left] {5}
             (1,1) circle (1pt) node[below] {-2}
             (2,1) circle (1pt) node[below] {-1}
             (3,1) circle (3pt) node[below] {0}
             (4,1) circle (3pt) node[below] {1}
             (5,1) circle (1pt) node[below] {2}
             (6,1) circle (1pt) node[below] {3}
             (7,1) circle (1pt) node[below] {4}
             (3,5) circle (3pt)
             (5,5) circle (3pt)
             (7,5) circle (3pt);
\end{tikzpicture}
$$
Here we label each contributed step $\mathfrak s$ by the number $\mathfrak s_y$, the $y$-coordinate of the starting point of $\mathfrak s$. Then the labels from the first path (the rightmost path in red color) corresponds to the first row of $\mathcal T$. Similarly, the labels from the second (in blue color) and third (in green color) paths give rise to the numbers in the second and third rows of $\mathcal T$, respectively. The content of the box filled with number $\mathfrak s_y$ equals $\mathfrak s_x$, the $x$-coordinate of the starting point of $\mathfrak s$. 

It is also very easy to write down the elements in the determinant $\det_{1\lle i,j\lle \la_1'}\mathscr S_{\la_i-\mu_j-i+j}(u+\mu_j-j+1)$ from the picture above. We order the starting points and end points from east to west using the set $\{1,\dots,\la_1'\}$. Then $\la_i-\mu_j-i+j$ corresponds to the horizontal length of any path from the $j$-th starting point to the $i$-th end point while $\mu_j-j+1$ is the $x$-coordinate of the $j$-th starting point. Therefore, we have
\[
\det_{1\lle i,j\lle \la_1'}\mathscr S_{\la_i-\mu_j-i+j}(u+\mu_j-j+1)=\left|\begin{matrix}
\mathscr S_3(u+1) & \mathscr S_4(u) & \mathscr S_6(u-2)\\
\mathscr S_1(u+1) & \mathscr S_2(u) & \mathscr S_4(u-2) \\
0 & 1 & \mathscr S_2(u-2)
\end{matrix}\right|,
\]where $0=\mathscr S_{-1}(u+1)$ means there is no path from the first starting point to the third end point and $1=\mathscr S_0(u)$ means there is exactly one path from the second starting point to the third end point. Moreover, this unique path contains no contributed steps.
\end{eg}

Clearly, it follows from Lemma \ref{lem chi morphism} that there are corresponding identities on the level of representations in the Grothendieck ring $\mathscr Rep(\mathcal C)$. Actually, Theorem \ref{thm Jacobi-Trudi} for the case of $\YglN$ has been shown in \cite{Che87,Che89} by resolutions of modules for the Yangian $\YglN$. Ignoring the spectral parameter $u$, one obtains the Jacobi-Trudi identity for super-characters of Lie superalgebra $\glMN$, see \cite{BB81}.

\begin{corollary}\label{cor zero}
If $\la/\mu$ contains a rectangle of size at least $(m+1)\times (n+1)$, then
\[
\det_{1\lle i,j\lle \la_1'}\mathscr S_{\la_i-\mu_j-i+j}(u+\mu_j-j+1)= \det_{1\lle i,j\lle \la_1}\mathscr A_{\la_i'-\mu'_j-i+j}(u-\mu'_j+j-1)=0.
\]
\end{corollary}
\begin{proof}
If $\la/\mu$ contains a rectangle of size at least $(m+1)\times (n+1)$, then there are no semi-standard Young tableaux of shape $\la/\mu$. Thus by Theorem \ref{thm character} we have $\mathscr K_{\la/\mu}(u)=0$. The statement now follows from Theorem \ref{thm Jacobi-Trudi}.
\end{proof}

\section{Drinfeld functor and skew representations}\label{sec: Drinfeld}
\subsection{Degenerate affine Hecke algebras}
In this section, we follow the exposition in \cite{Ara}. 

Let $l$ be a positive integer. Following \cite{Dri86}, the \emph{degenerate affine Hecke algebra} $\mathcal H_l$ is the associative algebra generated by generators $\sigma_1,\dots,\sigma_{l-1}$ and $x_1,\dots,x_l$ with the relations given by
\beq\label{eq:def-daha}
\begin{split}
\sigma_i^2=1,\quad \sigma_i\sigma_{i+1}\sigma_i=\sigma_{i+1}\sigma_i\sigma_{i+1},\quad [x_i,x_j]=0,\\
\sigma_ix_i=x_{i+1}\sigma_i-1,\quad [\sigma_j,\sigma_k]=[\sigma_j,x_k]=0 \quad \text{if }\ |j-k|\ne 1.
\end{split}
\eeq
As vector spaces, $\mathcal H_l\cong \C[\mathfrak S_l]\otimes \C[x_1,\dots,x_l]$. The generators $\sigma_1,\dots,\sigma_{l-1}$ generate a subalgebra isomorphic to $\C[\mathfrak S_l]$ while $x_1,\dots,x_l$ generate a subalgebra isomorphic to $\C[x_1,\dots,x_l]$. We shall use these identifications. It is well-known that the center of $\mathcal H_l$ is $\C[x_1,\dots,x_l]^{\mathfrak S_l}$.

Let $\sigma_{ij}$ be the simple permutation $(i,j)$. Let $y_1,\dots,y_l\in \mathcal H_l$ be defined by
\beq\label{eq y gen}
y_1=x_1,\qquad y_i=x_i-\sum_{j< i}\sigma_{ji},\qquad i=2,\dots,l.
\eeq
Then one has
\beq\label{eq y gen 1}
\sigma y_{i}=y_{\sigma(i)}\sigma,\qquad [y_i,y_j]=-(y_i-y_j)\sigma_{ij},
\eeq
for $\sigma\in \mathfrak S_l$ and $i,j=1,\dots,l$. Combining the relations in $\C[\mathfrak S_l]$, this could be considered as an alternate presentation of $\mathcal H_l$.

Consider the Lie algebra $\gl_N$ and its Cartan part $\h_N$. Let ${\bf \Phi}_N^+$ be the set of positive roots of $\gl_N$ and ${\bf P}_N$ the weight lattice of $\gl_N$. Let $\rho:=\sum_{i=1}^N(N-i)\epsilon_i\in \h_N$. Let
\[
{\bf D}_N^+:=\{\la\in \h_N^*~|~(\la+\rho)(\alpha)\notin \Z_{<0} \text{ for all }\alpha\in {\bf \Phi}_N^+\},\qquad {\bf P}_N^+:={\bf D}_N^+\cap {\bf P}_N.
\]
We call an element of ${\bf D}_N^+$ (resp. ${\bf P}_N^+)$ a {\it dominant} (resp. \textit{dominant integral}) $\gl_N$-weight. Recall that for a $\gl_N$-module $M$,  $\mathrm{wt}(M)$ denotes the set of all weights of $M$. In particular, we have
\[
\mathrm{wt}((\C^N)^{\otimes l})=\Big\{\sum_{i=1}^N n_i\epsilon_i ~|~n_i\in \Z_{\gge 0},\ \sum_{i=1}^N n_i=l\Big\}.
\]

Now we recall some facts from representation theory of degenerate affine Hecke algebra $\mathcal H_l$, following \cite{Zel, Rog}. 

Let $r\in \Z_{>0}$. Let $l=l_1+\dots+l_r$, where $l_i\in\Z_{\gge 0}$.
Then algebra $\mathcal H_{l_1}\otimes\cdots\otimes\mathcal H_{l_r}$ is identified with a subalgebra of $\mathcal H_l$ by the embedding
\[
\mathcal H_{l_k}\hookrightarrow\mathcal H_l,\qquad \sigma_a \mapsto \sigma_{a+l_1+\dots+l_{k-1}},\qquad x_b \mapsto x_{b+l_1+\dots+l_{k-1}},
\]
for $a=1,\dots,l_k-1$, $b=1,\dots,l_k$, and $k=1,\dots,r$.

For each $i=1,\dots,r$, let $M_i$ be an $\mc H_{l_i}$-module. Define the $\mc H_l$ module $M_1\odot \cdots \odot M_r$ via the induction functor:
\[
M_1\odot \cdots \odot M_r := \mathrm{Ind}_{\mathcal H_{l_1}\otimes\cdots\otimes\mathcal H_{l_r}}^{\mc H_l}(M_1\otimes \cdots \otimes M_r).
\]

For complex numbers $a$, $b$ such that $b-a+1=l$, denote by $\C_{[a,b]}:=\C {\bf 1}_{[a,b]}$ the one-dimensional representation of $\mathcal H_l$ given by
\begin{align}
&\sigma_i \cdot {\bf 1}_{[a,b]}={\bf 1}_{[a,b]}, & i=1,\dots,l-1,\label{eq sigma action}\\
&x_j \cdot {\bf 1}_{[a,b]}=(a+j-1){\bf 1}_{[a,b]}, & j=1,2,\dots,l.\label{eq x action}
\end{align}
For $\la\in \h_N^*$, set
\[
\mathscr W(\la;l):=\left\{\mu \in \h_N^*~|~\la-\mu \in \mathrm{wt}((\C^N)^{\otimes l})\right\}.
\]
Take $\mu\in \mathscr W(\la;l)$ and set $\la_i=\la(\epsilon_i)$, $\mu_i=\mu(\epsilon_i)$, and $l_i=\la_i-\mu_i$, for $i=1,\dots,N$. Define an $\mathcal H_l$-module by 
\beq\label{eq verma H}
\mathcal I(\la,\mu)= \C_{[\mu_1,\la_1-1]}\odot\C_{[\mu_2-1,\la_2-2]}\odot \cdots\odot \C_{[\mu_N-N+1,\la_N-N]}. 
\eeq
We also set $\mathcal I(\la, \mu)=0$ if $\mu\notin \mathscr W(\la;l)$.

When $\la$ is a dominant $\gl_N$-weight and $\mu\in \mathscr W(\la;l)$, we call $\mathcal I(\la,\mu)$ a \emph{standard module} of $\mathcal H_l$. The standard module can be thought as an analog of Verma modules.

Set ${\bf 1}_{\la,\mu}:={\bf 1}_{[\mu_1,\la_1-1]}\otimes  \cdots\otimes {\bf 1}_{[\mu_N-N+1,\la_N-N]}$, then there is an isomorphism of $\C[\mathfrak S_l]$-modules
\[
\mathcal I(\la,\mu)\cong \C[\mathfrak S_{l}/\mathfrak S_{l_1}\times \dots \times \mathfrak S_{l_N}]
\]
induced by ${\bf 1}_{\la,\mu}\mapsto 1$. In particular, one has the decomposition of $\mathfrak S_{l}$-modules,
\[
\mathcal I(\la,\mu)\cong \mathcal S(\nu_{\la,\mu})\oplus \bigoplus_{\nu > \nu_{\la,\mu}}\mathcal S(\nu)^{\oplus k_{\nu}},
\]
where $\nu_{\la,\mu}$ is the partition obtained by rearranging the sequence $(l_1,\dots,l_N)$ in non-increasing order and $>$ denotes the dominance order of partitions. Here $k_{\nu}$ are non-negative integers.

It is well-known, see \cite{Zel,Rog}, that if $\la\in {\bf D}_N^+$, then $\mathcal I(\la,\mu)$ is generated by the subspace $\mathcal S(\nu_{\la,\mu})$ over $\mathcal H_l$. Therefore, $\mathcal I(\la,\mu)$ has a unique irreducible quotient $\mathcal L(\la,\mu)$ containing $\mathcal S(\nu_{\la,\mu})$. The $\mathfrak S_l$-module $\mathcal S(\nu_{\la,\mu})$ appears in $\mathcal L(\la,\mu)$ (considered as an $\mathfrak S_l$-module) with multiplicity one. 

One has the following analog of the BGG resolution for $\mathcal H_l$-modules, see e.g. \cite[Proposition 7.3]{Suz} for a proof.
\begin{proposition}[\cite{Che87}]\label{prop res}
Let $\la\in \h^*_N$ and $\mu\in\mathscr W(\la;l)$. Suppose $\la-\rho\in {\bf P}_N^+$ and $\mu-\rho\in {\bf P}_N^+$, then there exists an exact sequence of $\mathcal H_l$-modules,
\[
0\to \bigoplus_{\sigma\in \mathfrak S_N[N(N-1)/2]}\mathcal I(\la,\sigma\cdot \mu)\to \cdots \to \bigoplus_{\sigma\in \mathfrak S_N[1]}\mathcal I(\la,\sigma\cdot \mu)\to \mathcal I(\la, \mu)\to \mathcal L(\la,\mu)\to 0,
\]
where $\mathfrak S_N[k]$ denotes the set of all elements of length $k$ in $\mathfrak S_N$ and $\sigma\cdot \mu:=\sigma(\mu+\rho)-\rho$ is the shifted Weyl group action.
\end{proposition}

\subsection{Drinfeld functor}\label{sec:Drinfeld-functor}
In this section, we define the Drinfeld functor \cite{Dri86} for the super Yangian $\YglMN$, following the exposition in \cite{Ara}.

Let $M$ be an $\mathcal H_l$-module. Consider the $\mathcal H_l\otimes \mathrm{U}(\gl_{m|n})$-module $M\otimes V^{\otimes l}$, where $V=\C^{m|n}\cong L(\epsilon_1)$ is the vector representation of $\glMN$. For $i=1,\dots,l$, let $$Q^{(i)}=(\mathcal P^{(0,i)})^{\top_0}=\sum_{a,b\in \bar I}(-1)^{|a||b|+|a|+|b|}E_{ab}\otimes E^{(i)}_{ab} \in \End(V)\otimes \End(V^{\otimes l}).$$ 
There is an algebra homomorphism
\begin{align*}
\wp: \ \YglMN &\to \mc H_l\otimes \End(V^{\otimes l}), \\
T(u) &\mapsto \mathscr T_1(u-x_1)\mathscr T_2(u-x_2)\cdots \mathscr T_{l}(u-x_l), 
\end{align*}
where
\[
\mathscr T_i(u-x_i)=1+\frac{1}{u-x_i}\otimes Q^{(i)}.
\]
Thus $M\otimes V^{\otimes l}$ becomes a $\YglMN$ module. One can think that it is a tensor product of evaluation $\YglMN$-modules with value in $M$, where the $i$-th copy of $V$ is evaluated at $x_i\in \mathcal H_l$, see \eqref{eq Hopf} and \eqref{eq:evaluation-map}.

The symmetric group acts naturally on $M\otimes V^{\otimes l}$ by $\sigma_{i}\mapsto \sigma_i\otimes \mathcal P^{(i,i+1)}$, $i=1,\dots,l-1$, where $\mathcal P^{(i,i+1)}\in \End(V^{\otimes l})$ is defined in \eqref{eq perm}. Set
\beq\label{eq d-functor}
\cD_{l}(M):=(M\otimes V^{\otimes l})/\sum_{i=1}^{l-1}\mathrm{Im}(\sigma_i+1),
\eeq
where $\mathrm{Im}(\sigma_i+1)$ denotes the image of $\sigma_i+1$ acting on $M\otimes V^{\otimes l}$.

\begin{lemma}
The subspace $\sum_{i=1}^{l-1}\mathrm{Im}(\sigma_i+1)\subset M\otimes V^{\otimes l}$ is an $\YglMN$-submodule. In particular, $\mathcal D_{l}(M)$ is a $\YglMN$-module.
\end{lemma}
\begin{proof}
The proof is parallel to \cite[Proposition 3]{Ara}. 
Since the coefficients of the polynomial $\prod_{i=1}^l(u-x_i)$ in $u$ belong to the center of $\mc H_l$, it suffices to check that $\big(\prod_{i=1}^l(u-x_i)\big) \wp(T(u))$ preserves the denominator space of \eqref{eq d-functor}. Applying the supertransposition $\top$ to the $0$-th factor of $\End(V)\otimes\End(V^{\otimes l})$, see \eqref{eq trans 2}, we obtain $[\cP^{(i,i+1)},Q^{(i)}]=[Q^{(i+1)},Q^{(i)}]$ from $[\cP^{(i,i+1)},\cP^{(0,i)}]=[\cP^{(0,i)},\cP^{(0,i+1)}]$. Using this equality and the defining relations \eqref{eq:def-daha}, one checks that
\[
(u-x_i+Q^{(i)})(u-x_{i+1}+Q^{(i+1)})\tl \sigma_i\equiv\tl\sigma_i(u-x_i+Q^{(i)})(u-x_{i+1}+Q^{(i+1)})+(\tl\sigma_i+1)[Q^{(i)},Q^{(i+1)}],
\]
where $\tl\sigma_i=\sigma_i\otimes \mc P^{(i,i+1)}$. Now the lemma follows.
\end{proof}

Denote by $\mathcal C_{\mc H_l}$ the category of finite-dimensional representations of $\mc H_l$. Recall that $\mc C$ is the category of finite-dimensional representations of $\YglMN$. The functor $\mc D_l$ is an exact functor from $\mathcal C_{\mc H_l}$ to $\mc C$. We call $\mc D_l$ the \emph{Drinfeld functor}, cf. \cite{Dri86}. 

In the case $n=0$ we recover the standard Drinfeld functor. In that case, we have the following useful well-known theorem.
We say that a representation of $\mathrm{Y}(\mathfrak{sl}_m)$ is \emph{of level $l$} if all its irreducible components when restricted as $\mathfrak{sl}_m$-modules are submodules of $(\C^m)^{\otimes l}$. Denote by $\mc C_{m}^{(l)}$ the category of finite-dimensional representations of $\mathrm{Y}(\mathfrak{sl}_m)$ with level $l$.

\begin{theorem}[{\cite{Dri86,CP}}]\label{thm equivalence}
If $l < m$, then the Drinfeld functor $\mathcal D_l$ is an equivalence between the category $\mc C_{\mc H_l}$ and the category $\mc C_{m}^{(l)}$. 
\end{theorem}

A supersymmetric version of Theorem \ref{thm equivalence} for quantum affine superalgebra $\mathrm U_q(\widehat{\mathfrak{gl}}_{m|n})$ was recently proved in \cite{Fli18} when $l < m+n$.

The following lemma describes the relation between $M$ considered as an $\mathfrak S_l$-module and $\cD_l(M)$ considered as a $\glMN$-module.

\begin{lemma}\label{lem duality}
Let $M$ be an $\mc H_l$-module. Let $M=\bigoplus_{\nu}\mc S(\nu')^{\oplus k_{\nu}}$ be the decomposition of $M$ as an $\mathfrak S_l$-module, where the sum is over all partitions of $l$ and $k_{\nu}\in\Z_{\gge 0}$. Then we have the decomposition of $\glMN$-modules,
\[
\mc D_l(M)\cong \bigoplus_{\nu\in \mathscr P_l(m|n)}L(\nu^\natural)^{\oplus k_{\nu}}.
\]
\end{lemma}
\begin{proof}
The statement follows from Theorem \ref{thm schur} (Schur-Sergeev duality), see the proof of \cite[Proposition 4]{Ara}.
\end{proof}

\begin{lemma}\label{lem tensor H}
Let $M_1$ be an $\mc H_{l_1}$-module and $M_2$ an $\mc H_{l_2}$-module. Then we have the $\YglMN$-module isomorphism
\[
\cD_{l_1}(M_1)\otimes \cD_{l_1}(M_2)\cong \cD_{l_1+l_2}(M_1\odot M_2).
\]
\end{lemma}
\begin{proof}
The proof is similar to that of \cite[Proposition 4.7]{CP} or \cite[Proposition 5.3]{Naz99}.
\end{proof}

\begin{lemma}\label{lem y act}
Let $M$ be an $\mc H_l$-module. Then the action of $\YglMN$ on $\mc D_l(M)$ can be written in the form
\[
t_{ij}(u)= \delta_{ij}+\sum_{k=1}^l \frac{1}{u-y_k}\otimes (-1)^{|i|} E_{ij}^{(k)},
\] where $y_k$ are given by \eqref{eq y gen}. Specifically, $t_{ij}^{(a)}$ acts by $\sum_{k=1}^l y_k^{a-1}\otimes (-1)^{|i|}E_{ij}^{(k)}$ for $a\gge 1$.
\end{lemma}
\begin{proof}
The proof is parallel to that of \cite[Proposition 6]{Ara}. We show by induction on $k$ that
\[
\Big(1+\frac{Q^{(1)}}{u-x_1}\Big)\cdots\Big(1+\frac{Q^{(k)}}{u-x_k}\Big)\equiv 1+\sum_{i=1}^k\frac{Q^{(i)}}{u-y_i}
\]
on $\mc D_l(M)$. The case $k=1$ is trivial. Using the induction hypothesis, we have
\begin{align*}
\Big(1+\frac{Q^{(1)}}{u-x_1}\Big)\cdots\Big(1+\frac{Q^{(k)}}{u-x_k}\Big)\equiv 1+\Big(\sum_{i=1}^{k-1}\frac{Q^{(i)}}{u-y_i}\Big) + \frac{1}{u-x_k}Q^{(k)}+  \sum_{i=1}^{k-1}\frac{1}{u-y_i} \frac{1}{u-x_k}Q^{(i)}Q^{(k)}.
\end{align*}
By applying supertransposition $\top$ to the $0$-th factor of $\End(V)\otimes\End(V^{\otimes l})$ of the identity $\cP^{(0,k)}\cdot \cP^{(0,i)}=\cP^{(i,k)}\cdot \cP^{(0,k)}$, one obtains $Q^{(i)}\cdot Q^{(k)}=\mc P^{(i,k)}\cdot Q^{(k)}$ . Using the relation $\sigma_{ik}=-\mc P^{(i,k)}$ on $\mc D_l(M)$ and \eqref{eq y gen}, \eqref{eq y gen 1}, we compute
\begin{align*}
\frac{1}{u-x_k}Q^{(k)}& +  \sum_{i=1}^{k-1}\frac{1}{u-y_i} \frac{1}{u-x_k}Q^{(i)}Q^{(k)}= \frac{1}{u-x_k}Q^{(k)} -\Big(\sum_{i=1}^{k-1}\sigma_{ik} \frac{1}{u-y_i} \frac{1}{u-x_k}Q^{(k)}\Big)\\
=&\  \frac{1}{u-y_k}\Big(u-y_k-\sum_{i=1}^{k-1}\sigma_{ik}\Big) \frac{1}{u-x_k}Q^{(k)} =\ \frac{1}{u-y_k}(u-x_k) \frac{1}{u-x_k}Q^{(k)}=\frac{1}{u-y_k}Q^{(k)}.
\end{align*}
\end{proof}

Define $\omega_k\in \h^*$, for each $k\in\Z_{>0}$, by
\[
\omega_k=\begin{cases} \epsilon_1+\cdots+\epsilon_{k}, & \text{ if }k\lle m;\\
\epsilon_1+\cdots+\epsilon_{m}+(k-m)\epsilon_{m+1},  & \text{ if }k\gge m.
\end{cases}
\]
Let $\la\in \h_N^*$ and $\mu\in \mathscr W(\la;l)$. Set $\la_i=\la(\epsilon_i)$, $\mu_i=\mu(\epsilon_i)$, and $l_i=\la_i-\mu_i$, for $i=1,\dots,N$. Recall the $\mc H_l$-module $\mc I(\la,\mu)$ defined in \eqref{eq verma H}.
\begin{lemma}\label{lem weyl Y}
Let $a,b$ be complex numbers such that $b-a+1=l$. Then we have the $\YglMN$-module isomorphism $\cD_l(\C_{[a,b]})\cong L_a(\omega^{}_l)$. 
Moreover, we have the $\YglMN$-module isomorphism
\[
\mc M(\la,\mu):=\mc D_l(\mc I(\la,\mu))\cong L_{\mu_1}(\omega^{}_{l_1})\otimes L_{\mu_2-1}(\omega^{}_{l_2})\otimes\cdots\otimes L_{\mu^{}_N-N+1}(\omega^{}_{l_N}).
\]
\end{lemma}
\begin{proof}
From \eqref{eq y gen}, \eqref{eq sigma action}, and \eqref{eq x action} we obtain that $y_i\cdot {\bf 1}_{[a,b]}=a{\bf 1}_{[a,b]}$. Therefore, the first statement follows from Lemma \ref{lem duality} and Lemma \ref{lem y act}. The second statement follows from the first statement and Lemma \ref{lem tensor H}.
\end{proof}

Unlike the $\gl_N$ case \cite[Theorem 9]{Ara}, when $mn>0$, the $\YglMN$-module $\mc M(\la,\mu)$ is never zero.

\subsection{Drinfeld functor and skew representations}
In this section, we study the relations between skew representations and Drinfeld functor. We need the following proposition.
\begin{proposition}\label{prop simple}
Let $M$ be a finite-dimensional irreducible representation of $\mc H_l$, then the $\YglMN$-module $\cD_l(M)$ is irreducible.
\end{proposition}
\begin{proof}
The proof is similar to that of \cite[Theorem 5.5]{Naz99}. We give here a brief account of the main points.

Let $v$ be a non-zero vector in  $\cD_l(M)$. We will show $v$ is cyclic over $\YglMN$. 
Without loss of generality, we can assume that $v$ is the image of $\sum_i w_i\otimes v_i\in M\otimes V^{\otimes l}$ such that $v_i$ are linearly independent and that $\sigma (\sum_i w_i\otimes v_i)=(-1)^\sigma (\sum_i w_i\otimes v_i)$ for $\sigma\in\fkS_l$.

Let $M_0$ be the $\fkS_l$-submodule of $M$ generated by all $w_i$. 
Since $M$ is irreducible over $\mc H_l$, $M$ is a quotient of the induced $\mc H_l$-module $\tl M=\mc H_l\otimes_{\fkS_l} M_0$ and it is enough to show that the image of $1\otimes (\sum_i w_i\otimes v_i)\in\tl M\otimes V^{\otimes l}$ in $\mc D_l(\tilde M)$ is a cyclic vector over $\YglMN$.

We have a filtration on $\mc H_l$ such that $\deg x_i=1$, $\deg \fkS_l=0$ which induces a filtration on $\cD_l(\tilde M)$. We also have a filtration on $\YglMN$ such that $\deg T_{ij}^{(s)}=s-1$. These two filtrations are compatible in the sense that $\cD_l(\tilde M)$ is a filtered $\YglMN$-module. We pass to the associated graded spaces. Then we get the space
$\tilde M^{\rm gr}=\C[x_1,\dots,x_l]\otimes M_0\otimes V^{\otimes l}$. The symmetric group is acting on the first factor in a natural way (that is  $\sigma x_i=x_{\sigma(i)}$) and on the other factors as before. The associated graded of the Yangian is the current algebra $\mathrm{U}(\glMN[t])$ acting trivially on $M_0$ and as on the product of evaluation vector representations with evaluation parameters $x_1,\dots,x_l$ on the rest.

Skew-invariants are vectors $v$ such that $\sigma v=(-1)^\sigma v$, $\sigma \in \fkS_l$.

It is enough to show that the skew-invariants of the $\mathrm{U}(\glMN[t])$-module $\tilde M^{\rm gr}$ are generated by $1\otimes (\sum_i w_i\otimes v_i)$. This follows from the standard fact that in endomorphisms of $\C[x_1,\dots,x_l]\otimes V^{\otimes l}$ the algebra
$\mathrm{U}(\glMN[t])$ coincides with the set of all supersymmetric operators of the form $\sum_{\sigma\in\fkS_l}\sigma p \otimes \sigma A \sigma^{-1} $, where $p\in\C[x_1,\dots,x_l]$ acts by multiplication and $A\in\End(V)^{\otimes l}$. Indeed, skew-invariants are spanned by $\sum_{\sigma\in \fkS_l} (-1)^\sigma \sigma(p_0\otimes w_i\otimes v_0)$ where $v_0\in V^{\otimes l}$, $p_0\in\C[x_1,\dots,x_l]$. We obtain this skew-invariant by choosing $p=p_0$ and any $A$ sending $v_j\to 0$ for $j\neq i$, and $v_i \to v_0$.
\end{proof}

We investigate the $\YglMN$-module $\cD_{l}(\mc L(\la,\mu))$. Note that the highest $\ell$-weight vector of $\cD_{l}(\mc L(\la,\mu))$ is not given by the quotient image of tensor product of highest $\ell$-weight vectors of all $L_{\mu_i-i+1}(\omega^{}_{l_i})$ in general and therefore the computation of highest $\ell$-weight of $\cD_{l}(\mc L(\la,\mu))$ is not straightforward. We use the resolution of the $\mathcal H_l$-module $\mc L(\la,\mu)$ and Jacobi-Trudi identity of $q$-characters to show that $\cD_{l}(\mc L(\la,\mu))$ is a skew representation if  $\la_{i}-\la_{i+1}\in \Z_{\gge 0}$ and $\mu_i-\mu_{i+1}\in \Z_{\gge 0}$ for all $i=1,\dots,N-1$.

For a partition $\la$ of length at most $N$, we identify $\la$ with a $\gl_N$-weight in the usual way.  We denote this $\gl_N$-weight also by $\la$. Clearly, $\la-\rho\in {\bf P}_N^+$.
\begin{theorem}\label{thm skew}
Let $\la$ and $\mu$ be partitions such that $\mu\in\mathscr W(\la;l)$. Then we have $\cD_l(\mc L(\la,\mu))\cong L(\la'/\mu')$ as $\YglMN$-modules. In particular, every skew representation is irreducible.
\end{theorem}
\begin{proof}
Applying the Drinfeld functor $\cD_l$ to the resolution of the $\mc H_l$-module $\mc L(\la,\mu)$ in Proposition \ref{prop res}, we have the exact sequence of $\YglMN$-modules,
\[
0\to \bigoplus_{\sigma\in \mathfrak S_N[N(N-1)/2]}\mathcal M(\la,\sigma\cdot \mu)\to \cdots \to \bigoplus_{\sigma\in \mathfrak S_N[1]}\mathcal M(\la,\sigma\cdot \mu)\to \mathcal M(\la, \mu)\to \cD_l(\mc L(\la,\mu))\to 0.
\]
For $\sigma\in\mathfrak S_N$, it is clear from Lemma \ref{lem chi morphism} and Lemma \ref{lem weyl Y} that
\[
\chi(\mc M(\la,\sigma\cdot\mu))=\prod_{i=1}^N\mathscr A_{\la_i-\mu_{\sigma^{-1}(i)}-i+\sigma^{-1}(i)}(u-\mu_{\sigma^{-1}(i)}+\sigma^{-1}(i)-1),
\]
where $N\gge \la_1'$. By the resolution above, we obtain that
\begin{align*}
\chi(\cD_l(\mc L(\la,\mu)))=&\ \sum_{\sigma\in\fkS_N}(-1)^{\sigma}\chi(\mathcal M(\la,\sigma\cdot \mu))
\\ =&\  \sum_{\sigma\in\fkS_N}(-1)^{\sigma}\prod_{i=1}^N\mathscr A_{\la_i-\mu_{\sigma(i)}-i+\sigma(i)}(u-\mu_{\sigma(i)}+\sigma(i)-1)
\\
=&\ \det_{1\lle i,j\lle \la_1'}\mathscr A_{\la_i-\mu_{j}-i+j}(u-\mu_{j}+j-1).
\end{align*}
It follows from Theorem \ref{thm Jacobi-Trudi} that 
$$
\chi(\cD_l(\mc L(\la,\mu)))= \mathscr  K_{\la'/\mu'}(u)=\chi(L(\la'/\mu')).
$$
Since by Proposition \ref{prop simple}, $\cD_l(\mc L(\la,\mu))$ is an irreducible $\YglMN$-module, we conclude again from Lemma \ref{lem chi morphism} that $\cD_l(\mc L(\la,\mu))\cong L(\la'/\mu')$. In particular, $L(\la'/\mu')$ is irreducible. The second statement follows from the fact that every partition can be identified with a $\gl_N$-weight if $N$ is sufficiently large.
\end{proof}

Give a partition $\la$ of length at most $N$ and a complex number $z$. Define $\gl_N$-weights $\la_z$ and $0_{z}$ by
\[
\la_z(\epsilon_i)=\la_i+z,\qquad 0_z(\epsilon_i)=z,\qquad i=1,\dots,N.
\]

\begin{corollary}\label{cor skew}
Suppose the same conditions as in Theorem \ref{thm skew} hold. Let $z$ be an arbitrary complex number, then we have $\cD_l(\mc L(\la_z,\mu_z))\cong L_z(\la'/\mu')$. 
\end{corollary}

\subsection{Fusion procedure and skew representations}
In this section, we study further skew representations by fusion procedure, following \cite{Che86,NT02,Naz04}.

Fix partitions $\la$ and $\mu$ such that $\mu\subset \la$. Set $l=|\la|-|\mu|$. Let $\Omega$ be a standard Young tableau of shape $\la/\mu$. For $i\in\{1,\dots,l\}$, denote the content of the box in $\Omega$ containing $i$ by $c_i(\Omega)$. Consider the operator
\[
\mathcal E_{\Omega}:=\mathop{\overrightarrow\prod}\limits_{1\lle i<j\lle l}R_{ij}\big(c_i(\Omega)-c_j(\Omega)\big)\in\End (V^{\otimes l}),
\]
where $V=\C^{m|n}$, $R_{ij}(u)=1-\mathcal P^{(i,j)}/u$ is the rational  R-matrix for $\YglMN$, and the order of the product corresponds to the writing of permutation $\sigma\in\mathfrak S_l$, $\sigma(i)=l+1-i$, in terms of simple transpositions.

It is known from \cite{Che86} that $\mathcal E_{\Omega}$ is well-defined, for a proof see e.g. \cite[Proposition 2.2]{NT02}. 

Consider the tensor product of evaluation $\YglMN$-modules $V_{z_1}\widetilde\otimes\cdots\widetilde\otimes V_{z_l}$ for $z_1,\dots,z_l\in \C$, where $\widetilde\otimes$ denotes the tensor product induced by the opposite coproduct \eqref{eq coop}. It follows from \eqref{eq coop} and \eqref{eq:evaluation-map} that the corresponding homomorphism $
\pi_{z_1,\dots,z_l}:\YglMN\to \End(V^{\otimes l})
$
is given by
\[
T(u)\mapsto R_{0,l}^{\top}(-u+z_l)\cdots R_{0,1}^{\top}(-u+z_1)
\]
under the homomorphism
\[
\mathrm{id}\otimes \pi_{z_1,\dots,z_l}:\End(V)\otimes \YglMN\to \End(V)\otimes \End(V^{\otimes l}),
\]
where $\top$ stands for the supertranspose \eqref{eq trans 2} on the $0$-th factor of $\End(V)$ in $\End(V)\otimes\End(V^{\otimes l})$. Denote by $\mathfrak P$ the linear operator on $V^{\otimes l}$ reversing the order of the tensor factors.

\begin{proposition}\label{prop fusion sub}
Set $z_i=-c_i(\Omega)$ for $i=1,\dots,l$. Then the operator $\mathcal E_{\Omega}\circ \mathfrak P$ is a $\YglMN$-module homomorphism
\[
\mathcal E_{\Omega}\circ \mathfrak P:V_{z_l}\widetilde \otimes\cdots\widetilde \otimes V_{z_1}\to V_{z_1}\widetilde \otimes\cdots\widetilde \otimes V_{z_l}.
\]In particular, the image of $\mathcal E_{\Omega}$ in $V^{\otimes l}$ is a submodule of the $\YglMN$-module $V_{z_1}\widetilde \otimes\cdots\widetilde \otimes V_{z_l}$.
\end{proposition}
\begin{proof}
The proof is similar to that of \cite[Proposition 4.2]{Naz04}. Note that our $z_i$ corresponds to $-z_i$ there. Explicitly, the proof is modified by applying the supertransposition $\top$ on the $0$-th copy of $\End(V)$ in $\End(V)\otimes\End(V)^{\otimes l}$ and replacing $x$ with $-u$, cf. \cite[Section 6.5]{Mol07}. 	
\end{proof}

We denote by $\mathscr F(\Omega)$ the $\YglMN$-submodule of $V_{z_1}\widetilde \otimes\cdots\widetilde \otimes V_{z_l}$ defined by the image of $\mathcal E_{\Omega}$ acting on $V^{\otimes l}$. Now we are ready to compare $\mathscr F(\Omega)$ with $L(\la/\mu)$.

Define the rational function $g_{\mu}(u)$ by
\[
g_{\mu}(u)=\prod_{i\gge 1}\frac{(u+\mu_i-i)(u-i+1)}{(u+\mu_i-i+1)(u-i)}.
\]
Then $g_{\mu}(\infty)=1$ and we identify $g_{\mu}(u)$ as a series in $\mathcal B=1+u^{-1}\C[[u^{-1}]]$. Recall the automorphism defined by $\Gamma_{\vartheta}:T(u)\mapsto \vartheta(u)T(u)$ from \eqref{eq mul auto} for $\vartheta(u)\in \mathcal B$. Denote by $\C_{\vartheta}^{(p)}$ the one-dimensional $\YglMN$-module of parity $p$ defined by the homomorphism
$$
\varepsilon\circ \Gamma_{g_\mu}: \YglMN \to \End(\C_{\vartheta}^{(p)}), \qquad T(u)\mapsto \vartheta(u),
$$
 where $\varepsilon$ is the counit map, see \eqref{eq Hopf}. We simply write $\C_{\vartheta}$ for the even case $\C_{\vartheta}^{(\bar 0)}$.

\begin{theorem}\label{thm equiv}
The $\YglMN$-modules $L(\la/\mu)\otimes \C_{g_\mu}$ and $\mathscr F(\Omega)$ are isomorphic.
\end{theorem}
\begin{proof}
The proof is similar to that of \cite[Theorem 1.6]{Naz04}. We only remark that the irreducibility of $L(\la/\mu)$ there is obtained from Olshanski's centralizer construction \cite{MO00} of $\YglN$. In this paper we obtain the irreducibility of $\YglMN$ module $L(\la/\mu)$ in a different way using Drinfeld functor, see Theorem \ref{thm skew}.
\end{proof}

We have the so-called binary property for tensor products of skew representations of $\YglMN$, cf. \cite[Theorem 4.9]{NT02}.

\begin{theorem}\label{thm:binary}
	Let $\la^{(i)}$ and $\mu^{(i)}$ be partitions such that $\mu^{(i)}\subset \la^{(i)}$, $i=1,\dots,k$. Let $z_1,\dots,z_k$ be complex numbers. Then the $\YglMN$-module $L_{z_1}(\la^{(1)}/\mu^{(1)})\otimes\cdots\otimes L_{z_k}(\la^{(k)}/\mu^{(k)})$ is irreducible if and only if $L_{z_i}(\la^{(i)}/\mu^{(i)})\otimes L_{z_j}(\la^{(j)}/\mu^{(j)})$ is irreducible for all $1\lle i< j\lle k$.
\end{theorem}
\begin{proof}
Let $\Omega_i$ be the column tableau of shape $\la_i/\mu_i$ for $i=1,\dots,k$. Denote by $\mathscr F_{z_i}(\Omega_i)$ the pull back of $\mathscr F(\Omega_i)$ through $\tau_{z_i}$. Thanks to Theorem \ref{thm equiv}, it suffices to show that
$$
\mathscr F_{z_1}(\Omega_1)\otimes \cdots\otimes \mathscr F_{z_k}(\Omega_k)
$$
is irreducible if and only if $\mathscr F_{z_i}(\Omega_i)\otimes\mathscr F_{z_j}(\Omega_j)$ for all $1\lle i < j\lle k$, see \cite[Theorems 4.8 and 4.9]{NT02}. The argument in \cite{NT02} using fusion procedure concerns the operators in the group algebra $\C[\mathfrak S_l]$ which can be generalized to the super setting with very few changes. Therefore the statement follows.
\end{proof}

For any $\YglMN$-modules $M_1,\dots,M_k$, we have
\[
(M_1\widetilde\otimes\cdots\widetilde\otimes M_k)^\iota\cong M_1^\iota\otimes\cdots\otimes M_k^\iota.
\]
Let $g^\iota_{\mu}(u)=g_{\mu}(-u)$. 

\begin{theorem}\label{thm * iso}
The $\YglMN$-modules $L(\la/\mu)^\iota\otimes \C_{g^\iota_\mu}$ and $\mathscr F(\Omega)^\iota$ are isomorphic and the $\YglMN$-module $\mathscr F(\Omega)^\iota$ is a submodule of $\mathbb V_{c_1(\Omega)} \otimes\cdots \otimes \mathbb V_{c_l(\Omega)}$, where $\mathbb V$ is the modified evaluation vector representation. 
\end{theorem}
\begin{proof}
The statement follows from Proposition \ref{prop fusion sub} and Theorem \ref{thm equiv}.
\end{proof}

\subsection{Application: irreducibility of tensor products}\label{sec:irr}
Due to Theorem \ref{thm equivalence} and Proposition \ref{prop simple},  many results from the representation theory of $\YglN$ can be generalized to the case of $\YglMN$. We give two such examples in this and next sections.

The following statement should be well-known for experts. However, we are not able to find the suitable reference, cf. \cite{Zel}.

\begin{proposition}\label{prop ind simple}
Let $\la^{(i)}$ and $\mu^{(i)}$ be partitions such that $\mu^{(i)}\subset \la^{(i)}$, $i=1,\dots,k$. Let $z_1,\dots,z_k$ be complex numbers such that $z_i-z_j\not\in \Z$ for all $1\lle i<j\lle k$. Then the induction product \[
\mc L(\la_{z_1}^{(1)},\mu_{z_1}^{(1)})\odot \cdots \odot\mc L(\la_{z_k}^{(k)},\mu_{z_k}^{(k)})
\]
is an irreducible $\mc H_l$-module, where $l=\sum_{i=1}^k(|\la^{(i)}|-|\mu^{(i)}|)$.
\end{proposition}
\begin{proof}
Let $N$ be sufficiently large. Applying the Drinfeld functor to 	$\mc L(\la_{z_1}^{(1)},\mu_{z_1}^{(1)})\odot \cdots \odot\mc L(\la_{z_k}^{(k)},\mu_{z_k}^{(k)})$ with $m=N$ and $n=0$, we obtain the $\YglN$-module
\[
L_{z_1}({\la^{(1)}}'/{\mu^{(1)}}')\otimes \cdots \otimes L_{z_k}({\la^{(k)}}'/{\mu^{(k)}}'),
\]
which is known to be irreducible when $z_i-z_j\not\in \Z$ for all $1\lle i<j\lle k$, see \cite[Corollary 3.9]{NT98}. The proposition follows from Theorem \ref{thm equivalence}.
\end{proof}

The following theorem is a direct corollary of Proposition \ref{prop simple}, Proposition \ref{prop ind simple}, and Corollary \ref{cor skew}.
\begin{theorem}
Let $\la^{(i)}$ and $\mu^{(i)}$ be partitions such that $\mu^{(i)}\subset \la^{(i)}$, $i=1,\dots,k$. Let $z_1,\dots,z_k$ be complex numbers such that $z_i-z_j\not\in \Z$ for all $1\lle i<j\lle k$. Then the tensor product of skew representations\[
L_{z_1}(\la^{(1)}/\mu^{(1)})\otimes \cdots \otimes L_{z_k}(\la^{(k)}/\mu^{(k)})
\]
is an irreducible $\YglMN$-module.	
\end{theorem}

Let $\la$ and $\mu$ be two partitions. Let $N$ be sufficiently large. Define the numbers
\[
a_i=\la_i-i+1, \qquad b_i=\mu_i-i+1,\qquad i=1,\dots,N.
\]
For each pair $(i,j)$ such that $1\lle i< j\lle N$, define the subsets of $\Z$ by
\[
\langle a_j,a_i\rangle =\{a_j,a_j+1,\dots,a_i\}\setminus \{a_j,a_{j+1},\dots,a_i\},
\]
\[
\langle b_j,b_i\rangle =\{b_j,b_j+1,\dots,b_i\}\setminus \{b_j,b_{j+1},\dots,b_i\}.
\]
Note that if $\la_i=\la_{i-1}=\cdots=\la_j$, then $\langle a_j,a_i\rangle =\O$.

\begin{proposition}\label{prop irr iff}
Let $\la$ and $\mu$ be two partitions, $z$ and $w$ two complex numbers. Then the induction product $\mc L((\la')_z,0_z)\odot \mc L((\mu')_w,0_w)$ is irreducible if and only if for each pair $(i,j)$ such that $1\lle i< j\lle N$, we have 
\beq\label{eq nonX}
b_j+z-w,\, b_i+z-w\notin \langle a_j,a_i\rangle \qquad \text{or}\qquad a_j-z+w,\, a_i-z+w\notin \langle b_j,b_i\rangle.
\eeq
In particular, $\mc L((\la')_z,0_z)\odot \mc L((\la')_z,0_z)$ is irreducible.
\end{proposition}
\begin{proof}
The proof is similar to that of Proposition \ref{prop ind simple} using \cite[Theorem 1.1]{Mol02}.	
\end{proof}

\begin{theorem}\label{thm irr suff cond}
	Let $\la$ and $\mu$ be two partitions, $z$ and $w$ two complex numbers. Suppose the condition \eqref{eq nonX} holds for all pairs $(i,j)$ such that $1\lle i< j$, then the $\YglMN$-module $L_z(\la^\natural)\otimes L_w(\mu^\natural)$ is irreducible. In particular, $L_z(\la^\natural)\otimes L_z(\la^\natural)$ is irreducible.
\end{theorem}
\begin{proof}
The theorem follows from Proposition \ref{prop simple} and Proposition \ref{prop irr iff}.
\end{proof}
Combining Theorem \ref{thm irr suff cond} with Theorem \ref{thm:binary}, one is able to give sufficient conditions for a tensor product of evaluation $\YglMN$-modules to be irreducible.

Comparing to the $\YglN$ case, conditions \eqref{eq nonX} are not necessary for $L_z(\la^\natural)\otimes L_w(\mu^\natural)$ to be irreducible. It would be interesting to generalize \cite[Theorem 1.1]{Mol02} to skew representations of $\YglMN$. 

\begin{eg}
We compare the sufficient and necessary conditions for $L_{z}(2\epsilon_1)\otimes L_{w}(2\epsilon_1)$ to be irreducible over $\mathrm{Y}(\gl_2)$ and $\mathrm{Y}(\gl_{1|1})$. The $\mathrm{Y}(\gl_2)$-module $L_{z}(2\epsilon_1)\otimes L_{w}(2\epsilon_1)$ is irreducible if and only if $z-w\ne \pm 1,\, \pm 2$, while the $\mathrm{Y}(\gl_{1|1})$-module $L_{z}(2\epsilon_1)\otimes L_{w}(2\epsilon_1)$ is irreducible if and only if $z-w\ne \pm 2$. Therefore the conditions \eqref{eq nonX} are not necessary for the irreducibility of the $\mathrm{Y}(\gl_{1|1})$-module $L_{z}(2\epsilon_1)\otimes L_{w}(2\epsilon_1)$.
\end{eg}

We call an irreducible $\YglMN$-module $M$ \emph{real} if $M\otimes M$ is also irreducible, see \cite{Lec03}. Theorem \ref{thm irr suff cond} implies that the evaluation module $L_z(\la^\natural)$ is real. Actually, it holds for all skew representations.

\begin{theorem}
The skew representation $L_z(\la/\mu)$ is real.
\end{theorem}
\begin{proof}
The statement follows from \cite[Remark (d) of Theorem 4.8]{NT02}, \cite[Theorem 1.6]{Naz04}, Theorem \ref{thm equivalence}, and Proposition \ref{prop simple}. 
\end{proof}

\subsection{Application: extended T-systems}\label{sec:T}
In this section, we apply Drinfeld functor to show that the $q$-characters of skew representations satisfy extended T-systems.

Let $\mho=\la/\mu$ be a skew Young diagram. We say that $\mho$ is a {\it prime} skew Young diagram if it can not be divided into two parts intersecting at most one point.

\begin{eg}
We explain the definition with the following 3 skew Young diagrams.
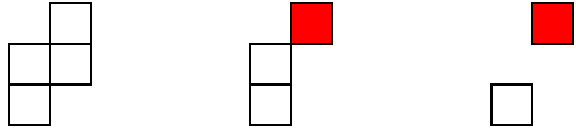
\[
\ytableaushort[*(white)]
{\none ~,~~,~\none}\qquad\qquad\qquad \ytableaushort[*(white)]
{\none{*(red)},~\none,~\none} \qquad\qquad\qquad \ytableaushort[*(white)]
{\none{*(red)},\none\none,~\none}
\]

\vskip0.2cm

The first skew Young diagram is prime while the rest are not. For example, the last two diagrams can be divided into two parts so that one part is in red color. Clearly, the two parts of the second one intersects at a point while those of the third one are disconnected.
\end{eg}

Recall that different pairs $(\la,\mu)$ may give the same skew Young diagram. Let $\mho$ be a prime skew Young diagram (ignoring the content). We choose a $\la$ so that  $\la$ and $\mho$ have the same number of columns and $\mho=\la/\mu$. The contents of $\mho$ are determined by $\la$, namely the box of $\la$ at the left-upper corner has content zero. Let $l$ be the number of columns of $\mho$, then $l=\la_1$.

Suppose $\mho=\la/\mu$ has at least two columns, namely $l\gge 2$. Let $\mho^+$ and $\mho^-$ be the prime skew Young diagrams obtained by deleting the leftmost column and the rightmost column of $\mho$, respectively. Also let $\mho^0$ be the prime skew Young diagrams obtained by removing both the leftmost and rightmost columns of $\mho$. Note that $\mho^0$ may be empty. 

Define two skew Young diagrams $\mathbb X_{\mho}$ and $\mathbb Y_{\mho}$ as follows,
\[
\mathbb X_{\mho}=\{(i,j): \mu'_{j}+1\lle  i \lle \la'_{j+1}-1, \ 1\lle j \lle l-1\},
\]
\[
\mathbb Y_{\mho}=\{(i,j): \mu'_{j+1}\lle  i \lle \la'_{j}, \ 1\lle j \lle l-1\}.
\]
The skew Young diagrams $\mathbb X_{\mho}$ and $\mathbb Y_{\mho}$ are  obtained by taking the intersection and union, respectively, of the diagram $\mho^+$ shifted to the left by one unit and then up by one unit and $\mho^-$.

Note that in general as $\mho$ is prime, we have $\mu'_j\lle \la'_{j+1}-1$ for $i=1,\dots,l-1$. Hence the $j$-th column of $\mathbb X_{\mho}$ may be empty and $\mathbb X_{\mho}$ may be non-prime. However, $\mathbb Y_{\mho}$ is always prime. 

Snakes defined in \cite{MY12a} bijectively correspond to certain skew Young diagrams via the correspondence in \cite[Proposition 7.3]{MY12a}. The skew Young diagrams $\mathbb X_{\mho}$ and $\mathbb Y_{\mho}$ correspond to the neighbouring snakes in \cite[Section 3.6]{MY12b} in this sense. 

Recall, if a skew Young diagram contains a rectangle of size $(m+1)\times (n+1)$ (a column of length $N+1$ in the $\YglN$ case), then the corresponding skew representation has dimension zero. 

\begin{theorem}[{\cite[Theorem 4.1]{MY12b}}]\label{thm T even}
Suppose $\mho$ is a prime skew Young diagram having at least two columns and $N$ is sufficiently large. Then we have the following relation in $\mathscr Rep(\YglN)$, the Grothendieck ring of the category of finite-dimensional representations of $\YglN$,
\[
[L(\mho^+)]\otimes [L(\mho^-)]=[L(\mho^0)][L(\mho)]+[L(\mathbb X_{\mho})][L(\mathbb Y_{\mho})].
\]Moreover, $L(\mho^0)\otimes L(\mho)$ and $L(\mathbb X_{\mho})\otimes L(\mathbb Y_{\mho})$ are irreducible $\YglN$-modules.
\end{theorem}

\begin{rem}
Because we only care about the case when $N$ is sufficiently large, our definitions of prime diagrams and $\mathbb Y_{\mho}$ here are slightly different from that in  \cite[Section 3.5]{MY12b} as we allow a column of a prime skew Young diagram or $\mathbb Y_{\mho}$ to be very long. 
\end{rem}

Applying Drinfeld functor, we get the corresponding supersymmetric version of Theorem \ref{thm T even}.

\begin{corollary}\label{cor T super}
Suppose $\mho$ is a prime skew Young diagram having at least two columns. Then we have the following relation in $\mathscr Rep(\mathcal C)$,
\[
[L(\mho^+)]\otimes [L(\mho^-)]=[L(\mho^0)][L(\mho)]+[L(\mathbb X_{\mho})][L(\mathbb Y_{\mho})].
\]Moreover, $L(\mho^0)\otimes L(\mho)$ and $L(\mathbb X_{\mho})\otimes L(\mathbb Y_{\mho})$ are irreducible $\YglMN$-modules.
\end{corollary}
\begin{proof}
Using Theorem \ref{thm T even} and Theorem \ref{thm equivalence} (the equivalence of Drinfeld functor), then we have the corresponding equality for the representations of degenerate affine Hecke algebra. Applying Drinfeld functor to this resulted equality, the first statement follows. Similarly, the second part follows from Proposition \ref{prop simple}.
\end{proof}

\begin{eg}
Given $i,j\in\Z_{>0}$ and $k\in \C$, let $\mho_{ij;k}$ be the rectangular Young diagram of size $i\times j$ whose left-upper corner box has content $k$. Then we have 
$$\mho_{i(j+1);k}^+=\mho_{ij;k+1},\quad \mho_{i(j+1);k}^-=\mho_{ij;k},\quad \mho_{i(j+1);k}^0=\mho_{i(j-1);k+1}, 
$$
$$
\mathbb X_{\mho_{i(j+1);k}}=\mho_{(i-1)j;k},\qquad \mathbb Y_{\mho_{i(j+1);k}}=\mho_{(i+1)j;k+1}.
$$

Let $$
\mathbf T_j^{(i)}(u+k-(i-j+1)/2)= \mathscr K_{\mho_{ij;k}}(u),
$$
where $\mathscr K_{\mho_{ij}}(u)$ is the $q$-character of $L(\mho_{ij;k})$, see Theorem \ref{thm character}.

Setting $k=(i-j)/2$ and $\mho=\mho_{i(j+1);k}$, one obtains the T-systems from Corollary \ref{cor T super},
\[
\mathbf T_j^{(i)}(u-\tfrac{1}{2})\mathbf T_j^{(i)}(u+\tfrac{1}{2})=\mathbf T_{j-1}^{(i)}(u)\mathbf T_{j+1}^{(i)}(u)+\mathbf T_{j}^{(i-1)}(u)\mathbf T_{j}^{(i+1)}(u).
\]
The boundary conditions are given by
\begin{itemize}
    \item $\mathbf T_{j}^{(i)}(u)=0$ if $i<0$ or $j<0$ or both $i>m$ and $j>n$;
    \item $\mathbf T_{j}^{(i)}(u)=1$ if $i,j\in \Z_{\gge 0}$ and $ij=0$.
\end{itemize} 
Hence we may regard Corollary \ref{cor T super} as extended T-systems.
\end{eg}

We remark that our extended T-systems are different from that of \cite[Theorem 3.3]{Zhh18}.

Suppose $\mho$ is a skew Young diagram with consecutive columns (not necessarily prime) and has at least 2 columns. One can define $\mho^\pm$ and $\mho^0$ in the same way.
\begin{theorem}
If $\mho$ is not prime, then 
$ L(\mho^+)\otimes L(\mho^-)$ and  $L(\mho^0)\otimes L(\mho)$ are isomorphic and irreducible as $\YglMN$-modules.
\end{theorem}
\begin{proof}
The theorem is proved in a similar way to that of Corollary \ref{cor T super} using \cite[Theorem 4.3]{MY12b}.
\end{proof}

\section{Quantum Berezinian and transfer matrices}\label{sec transfer matrix}
\subsection{Quantum Berezinian}
Following \cite{MR14}, we recall the quantum Berezinian and related results in the case of $\YglMN$.

Let $\mathcal A$ be a superalgebra. Let $a_{ij}\in \mathcal A$ with parity $|i|+|j|$. Suppose the inverse of the matrix
\[
A=\sum_{i,j\in \bar I}a_{ij}\otimes E_{ij} (-1)^{|i||j|+|j|}\in \mathcal A\otimes \End(V),
\]
with values in $\mathcal A$ exists. Then we denote the entries of the inverse matrix by $a'_{ij}\in \mathcal A$:
\[
A^{-1}=\sum_{i,j\in \bar I}a'_{ij}\otimes E_{ij} (-1)^{|i||j|+|j|}.
\] 
Define the \emph{quantum Berezinian} $\mathrm{Ber}(A)$, see \cite{Naz}, of the matrix $A$ by
\[
\mathrm{Ber}(A)=\sum_{\sigma\in\mathfrak S_m}\mathrm{sgn}(\sigma)\cdot a_{\sigma(1)1}\cdots a_{\sigma(m)m}\sum_{\tl\sigma\in\mathfrak S_n}\mathrm{sgn}(\tl\sigma)\cdot a'_{m+1,m+\tl\sigma(1)}\cdots a'_{m+n,m+\tl\sigma(n)}.
\]  

Let  $\mathcal A_{m|n}:=\YglMN[[u^{-1}]]((\tau))$ be the superalgebra of Laurent series in $\tau$ whose coefficients are power series in $u^{-1}$ whose coefficients are in $\YglMN$ with the relations
$$
(g_1 u^{k_1} \tau^{l_1}) (g_2u^{k_2}\tau^{l_2})= g_1g_2 u^{k_1}(u-l_1)^{k_2} \tau^{l_1+l_2}, \qquad g_1,g_2\in\YglMN, l_1,l_2\in\Z,\ k_1,k_2\in\Z_{\leqslant 0}.
$$

Thus $\tau$ is the shift operator with respect to variable $u$ and it should not be confused with automorphism of the Yangian $\tau_1$ defined in \eqref{eq tau z}.

Let $q$ be a formal variable which commutes with all other elements. Let $A(q)$ be a matrix with elements in $\mathcal A_{m|n}[q]$ given by
$$
A(q)=1-q\, T(u)\tau=\sum_{i,j\in \bar I}(\delta_{ij}-q\,t_{ij}(u)\tau)\otimes E_{ij} (-1)^{|i||j|+|j|}.
$$
Clearly $A(q)$ is invertible. Let 
\begin{align}\label{ber}
\mathfrak D(u,\tau;q):=\mathrm{Ber}(A(q))
\end{align}
be the quantum Berezinian. We simply write $\mathfrak D(u,\tau)$ for $\mathfrak D(u,\tau;1)$. 

The matrix $T(u)$ is also invertible. Let $\mathfrak Z(u)=\mathrm{Ber}(T(u)\tau)\tau^{n-m}$. Note that  $\mathfrak Z(u)$  does not depend on $\tau$. It is known the coefficients of $\mathfrak Z(u)$ generate the center of $\YglMN$ and 
$$
\mathfrak Z(u)=\prod_{i\in \bar I}(d_i(u-\kappa_i))^{s_i},
$$
in the notation of Proposition \ref{prop center}, see \cite[Theorem 1]{Gow05} and \cite[Theorem 4]{Gow07}.

Recall the standard action of symmetric group $\fkS_{k}$ on the space $V^{\otimes k}$ where $\sigma_i$ acts as the graded flip operator $\cP^{(i,i+1)}$, see \eqref{eq perm}. We denote by $\mathbb A_k$ and $\mathbb S_k$ the images of the normalized anti-symmetrizer and symmetrizer, respectively,
\[
\mathbb A_k=\frac{1}{k!}\sum_{\sigma\in\fkS_k}\mathrm{sgn}(\sigma)\cdot \sigma,\qquad \mathbb S_k=\frac{1}{k!}\sum_{\sigma\in\fkS_k} \sigma.
\]
\begin{theorem}[\protect{\cite[Theorem 2.13]{MR14}}]\label{thm ber expansion}
We have
\begin{align}
&\mathfrak D(u,\tau;q)=1+\sum_{k=1}^{\infty}(-1)^k\mathrm{str}\,\mathbb A_k\,T_1(u)T_2(u-1)\cdots T_k(u-k+1)q^k\tau^k,\label{eq expand}\\
&\mathfrak D(u,\tau;q)^{-1}=1+\sum_{k=1}^{\infty}\mathrm{str}\,\mathbb S_k\,T_1(u)T_2(u-1)\cdots T_k(u-k+1)q^k\tau^k,\nonumber
\end{align}	
where the supertrace is taken over all copies of $\End(V)$.	
\end{theorem}

\subsection{Universal R-matrix and transfer matrices}
The Yangian $\YglMN$ has a universal R-matrix. We presume its existence and properties, cf. \cite{Naz20} and \cite{Dri85}. We do not provide any justification in this paper.

The universal R-matrix is an element $\mathscr R(u)\in 1+u^{-1}\YglMN\otimes \YglMN[[u^{-1}]]$ such that for all $X\in \YglMN$ we have
\[
(\mathrm{id}\otimes\Delta)(\mathscr R(u))=\mathscr R_{12}(u)\mathscr R_{13}(u)\in \YglMN^{\otimes 3}[[u^{-1}]],
\]
\beq\label{eq:delta-R}
(\Delta\otimes \mathrm{id})(\mathscr R(u))=\mathscr R_{13}(u)\mathscr R_{23}(u)\in \YglMN^{\otimes 3}[[u^{-1}]],
\eeq
\[
\mathscr R(u)\cdot (\mathrm{id}\otimes \tau_u)(\Delta^{\rm op}(X))=(\mathrm{id}\otimes \tau_u)(\Delta (X))\cdot \mathscr R(u)\in \YglMN^{\otimes 2}[[u^{-1}]],
\]
where $\tau_u$ is the Yangian automorphism defined in \eqref{eq tau z} (not to be confused with the shift operator $\tau$). It follows that the universal R-matrix $\mathscr R(u)$ satisfies the Yang-Baxter equation
\[
\mathscr R_{12}(u-v)\mathscr R_{13}(u)\mathscr R_{23}(v)=\mathscr R_{23}(v)\mathscr R_{13}(u)\mathscr R_{12}(u-v).
\]
Let $M$ be a finite-dimensional $\YglMN$-module. Denote by $\Theta_{M}:\YglMN\to \End(M)$ the corresponding map. 
We also assume that
\beq\label{eq:norm-R}
(\Theta_{\mathbb V_z}\otimes\mathrm{id})(\mathscr R(-u))=T(u+z)\in \End(V)\otimes \YglMN[[u^{-1}]],
\eeq
cf. \cite[Lemma 3.4 and Theorem 3.6]{Naz98}, \cite[Proposition 4.5]{Naz99}, and \cite[Proposition 13.5]{Naz20}.

Define the \emph{transfer matrix} $\mathfrak T_M$ associated to $M$ by
\[
\mathfrak T_M(u)=\str_M\big((\Theta_{M}\otimes \mathrm{id})(\mathscr R(-u))\big)\in \YglMN[[u^{-1}]].
\]

The following lemma is standard, see \cite[Lemma 2]{FR99}.
\begin{lemma}\label{lem trans relations}
For any pair of finite-dimensional $\YglMN$-modules $M_1$ and $M_2$, we have 
$$
[\mathfrak T_{M_1}(u_1),\mathfrak T_{M_2}(u_2)]=0,\qquad \mathfrak T_{M_1\otimes M_2}(u)=\mathfrak T_{M_1}(u)\mathfrak T_{M_2}(u).
$$ For a short exact sequence $M_1\hookrightarrow M \twoheadrightarrow M_2$, we have $\mathfrak T_{M}(u)=\mathfrak T_{M_1}(u)+\mathfrak T_{M_2}(u)$.
\end{lemma}

Lemma \ref{lem trans relations} says that the map $\mathfrak T:\mathscr Rep(\mathcal C) \to \YglMN[[u^{-1}]] $ sending a finite-dimensional $\YglMN$-module $M$ to the transfer matrix $\mathfrak T_{M}(u)$ in $\YglMN[[u^{-1}]]$ is a ring homomorphism.

We shall focus on transfer matrices associated to skew representations $L(\la/\mu)^\iota$. When $M=L(\la/\mu)^\iota$, we write $\mathfrak T_{\la/\mu}(u)$ for $\mathfrak T_M(u)$. Then we have $M_{-z}\cong (L_z(\la/\mu))^\iota$ and 
\beq\label{eq:shift-content-T}
\mathfrak T_{M_{-z}}(u)=\mathfrak T_{\la/\mu}(u-z).
\eeq
Recall that the partition $(1^k)$ corresponds to the Young diagram consisting of a column with $k$ boxes while $(k)$ corresponds to the Young diagram consisting of a row with $k$ boxes. We use the short-hand notation, 
$$
\mathfrak T^k(u):=\mathfrak T_{(1^k)}(u),\qquad \mathfrak T_k(u):=\mathfrak T_{(k)}(u).
$$
We have the Jacobi-Trudi identity for transfer matrices.

\begin{theorem}\label{thm Jacobi T}
Let $\la$ and $\mu$ be two partitions such that $\mu\subset \la$. Then we have
\begin{align*}
\mathfrak T_{\la/\mu}(u)=&\det_{1\lle i,j\lle \la_1'}\mathfrak T_{\la_i-\mu_j-i+j}(u+\mu_j-j+1)\\=&\det_{1\lle i,j\lle \la_1}\mathfrak T^{\la_i'-\mu_j'-i+j}(u-\mu_j'+j-1).
\end{align*}
Here we use the convention that $\mathfrak T_0(u)=\mathfrak T^0(u)=1$ and $\mathfrak T^k(u)=\mathfrak T_k(u)=0$ for $k<0$.
\end{theorem}
\begin{proof}
Since the $q$-character map is injective, see Lemma \ref{lem commute yangian}, Theorem \ref{thm Jacobi-Trudi} implies the corresponding equalities in the Grothendieck ring $\mathscr Rep(\mathcal C)$. Note that $\iota:\YglMN\to \YglMN^{\mathrm{op}}$ is a Hopf superalgebra isomorphism and $\mathscr Rep(\mathcal C)$ is commutative. Applying $\iota$ we obtain new equalities in $\mathscr Rep(\mathcal C)$ where all participating modules are twisted by $\iota$. Finally, the statement follows by further applying the transfer matrix homomorphism $\mathfrak T$.
\end{proof}

Theorem \ref{thm Jacobi T} was conjectured in \cite{Tsu:1997} on the level of eigenvalues, and proved for the case of hook Young diagrams in \cite{KV:2008}.
\begin{corollary}\label{cor zero T}
If $\la/\mu$ contains a rectangle of size at least $(m+1)\times (n+1)$, then
\[
\det_{1\lle i,j\lle \la_1'}\mathfrak T_{\la_i-\mu_j-i+j}(u+\mu_j-j+1)= \det_{1\lle i,j\lle \la_1}\mathfrak T^{\la_i'-\mu'_j-i+j}(u-\mu'_j+j-1)=0.
\]
\end{corollary}
\begin{proof}
The statement follows from Corollary \ref{cor zero}.
\end{proof}

\begin{proposition}\label{prop anti}
We have	
$$
\mathfrak T^k(u)=\mathrm{str}\,\mathbb A_k\,T_1(u)T_2(u-1)\cdots T_k(u-k+1),
$$  
$$
\mathfrak T_k(u)=\mathrm{str}\,\mathbb S_k\,T_1(u)T_2(u+1)\cdots T_k(u+k-1).
$$
\end{proposition}
\begin{proof} Consider the weight $\omega_k$. Then the corresponding
Young diagram is a column with $k$ boxes. Let $\Omega$ be the column tableau. Then $c_i(\Omega)=1-i$, $i=1,\dots,k$, and it is well-known that $\mathcal E_{\Omega}=k!\mathbb A_k$ and $\mathbb A_k$ is the projection of $V^{\otimes k}$ onto the image of $\mathcal E_{\Omega}$, which is isomorphic to $L(\omega_k)$ as a $\glMN$-module. It follows from Theorem \ref{thm * iso} that 
\[
\Theta_{L((1^k))^\iota}=\mathbb A_k\circ \big(\Theta_{\mathbb V_{c_1(\Omega)}}\otimes\cdots\otimes \Theta_{\mathbb V_{c_k(\Omega)}}\big) \circ \Delta^{(k-1)}.
\]
The first formula of the proposition now follows from \eqref{eq:delta-R} and \eqref{eq:norm-R}. The second formula is similar.
\end{proof}

\subsection{Harish-Chandra homomorphism}\label{sec:HC}
In this section, we define an analog of Harish-Chandra homomorphism $\mathscr H$ for $\YglMN$ and compute the images of transfer matrices associated to skew Young diagrams under $\mathscr H$. See \cite{MM14} for a discussion of the Harish-Chandra homomorphism for Yangians of classical Lie algebras.

Let $\YglMN^\h$ be the centralizer of $\h\subset \glMN$ in $\YglMN$,
\[
\YglMN^\h=\{X\in \YglMN\ \ |\ \  [t_{ii}^{(1)},X]=0,\ \text{ for } i\in \bar I\}.
\]
Recall $d_i^{(r)}$, $e_j^{(r)}$, and $f_j^{(r)}$ for $i\in \bar I$ and $j\in I$, from Section \ref{sec:Gauss}. Let $J$ be the intersection of $\YglMN^\h$ and the right ideal of $\YglMN$ generated by $f_j^{(r)}$, for $j\in I$ and $r\in\Z_{>0}$. Note that $J$ is also the intersection of $\YglMN^\h$ and the left ideal of $\YglMN$ generated by $e_j^{(r)}$, for $j\in I$ and $r\in\Z_{>0}$. We have the decomposition as vector spaces,
\[
\YglMN^\h=\mathrm Y_{m|n}^0 \oplus J.
\]
The projection $\mathscr H$ of $\YglMN^\h$ onto the subspace $\mathrm Y_{m|n}^0$ along $J$,  
$$
\mathscr H:\YglMN^\h\to \mathrm Y_{m|n}^0,
$$ 
is an algebra homomorphism. We call $\mathscr H$ the \emph{Harish-Chandra homomorphism} of $\YglMN$.

Clearly, from the Gauss decomposition, we have $\mathscr H(t_{ii}(u))=d_i(u)$. 

The following lemma can be proved using induction, cf. \cite[Lemma 4.1]{MTV06}.
\begin{lemma}\label{lem hc proj}
For any $X\in \YglMN^\h$ of the form $$*\, t_{i_0j_0}^{(r_0)}t_{i_1i_1}^{(r_1)}\cdots t_{i_ki_k}^{(r_k)},\quad i_a,j_0\in \bar I, \, r_a>0,\, a=0,1,\dots,k, \, k\in\Z_{\gge 0},\quad i_0<j_0,$$ where $*$ is an element in $\YglMN$, we have $\mathscr H(X)=0$.
\end{lemma}

It is well-known that all coefficients of $\mathfrak T^k(u)$ and $\mathfrak T_k(u)$ belong to the centralizer $\YglMN^\h$, cf. \cite[Proposition 4.7]{MTV06}. Hence we can compute the Harish-Chandra images of transfer matrices $\mathfrak T_{\la/\mu}(u)$.

\begin{lemma}\label{lem HC image AS}
We have
\[
\mathscr H(\mathfrak T^k(u))=\sum_{\mathscr I} \prod_{a=1}^k s_{i_a}d_{i_a}(u-a+1),
\]
\[
\mathscr H(\mathfrak T_k(u))=\sum_{\mathscr J} \prod_{a=1}^k s_{j_a}d_{j_a}(u-a+k),
\]
summed over all sequences $\mathscr I=\{1\lle i_1<i_2<\dots<i_b<m+1\lle i_{b+1}\lle \cdots\lle i_k\lle m+n\}$ and $\mathscr J=\{m+n\gge j_1> j_2>\dots>j_b\gge m+1> j_{b+1}\gge \cdots\gge j_k\gge 1\}$ with $b=0,\dots,k$, respectively.
\end{lemma}
\begin{proof}
The lemma follows from Proposition \ref{prop anti}, Lemma \ref{lem hc proj}, and an analog of \cite[Proposition 2.3]{MR14} with sequences of the form like $\mathscr I$ or $\mathscr J$, see \cite[Remark 2.4]{MR14}. A similar computation was done in \cite[Section 3.3]{MR14}.
\end{proof}

The following is immediate from Theorem \ref{thm ber expansion}
 and Lemma \ref{lem HC image AS}.
\begin{corollary}\label{cor hc image ber}
The Harish-Chandra image of $\mathfrak D(u,\tau;q)$ is given by
\[
\mathscr H(\mathfrak D(u,\tau;q))=\mathop{\overrightarrow\prod}\limits_{1\lle i\lle m+n}\Big(1-q\,d_i(u)\, \tau\Big)^{s_i}.
\]
\end{corollary}

\begin{proposition}\label{prop hc image skew}
We have 
\beq\label{eq hc images}
\mathscr H(\mathfrak T_{\la/\mu}(u))=\sum_{\mathcal T}\prod_{(i,j)\in\la/\mu}s_{\mathcal T(i,j)}d_{\mathcal T(i,j)}(u+c(i,j)),
\eeq
where the summation is over all semi-standard Young tableaux $\cT$ of shape $\la/\mu$.
\end{proposition}
\begin{proof}
The statement follows from the fact that $\mathscr H$ is an algebra homomorphism, Theorem \ref{thm Jacobi T}, Lemma \ref{lem HC image AS} and the proof of Theorem \ref{thm Jacobi-Trudi} (identifying $s_id_i(u+a)$ with $\mathscr X_{i,a}$).
\end{proof}

\begin{rem}
Note that if we identify $s_id_i(u+a)$ with $\mathscr X_{i,a}$ for $i\in\bar I$ and $a\in \C$, where $s_i$ is the parity of $\mathscr X_{i,a}$, then the right hand side of \eqref{eq hc images} is identified with right hand side of the equation in Theorem \ref{thm character}. This implies that the $q$-character map can be thought as the composition of Harish-Chandra map and the map $\mathfrak T$, see \cite[Section 3]{FR99}.
\end{rem}

\begin{corollary}\label{cor transfer nonzero}
If $\la/\mu$ does not contain a rectangle of size $(m+1)\times (n+1)$, then $\mathscr H(\mathfrak T_{\la/\mu}(u))$ is non-zero. In particular, $\mathfrak T_{\la/\mu}(u)$ is non-zero.
\end{corollary}
\begin{proof}
As $\la/\mu$ does not contain a rectangle of size $(m+1)\times (n+1)$, there exists at least one semi-standard Young tableau of shape $\la/\mu$. Hence the space $L(\la/\mu)$ is non-trivial by Theorem \ref{thm character} and it is an irreducible finite-dimensional $\YglMN$-module by Theorem \ref{thm skew}. Let $\mathcal T_0$ be the semi-standard Young tableau corresponding to the highest $\ell$-weight of $L(\la/\mu)$, see Theorem \ref{thm character}. We consider the monomial
$$
\prod_{(i,j)\in\la/\mu}\Big(d_{\mathcal T_0(i,j)}^{(1)}u^{-1}\Big)
$$
in the right hand side of \eqref{eq hc images}. The monomial $\mathcal T_0$ corresponds to the highest $\gl_{m|n}$-weight in $L(\la/\mu)$. Clearly this monomial appears only in 
$$
\prod_{(i,j)\in\la/\mu}s_{\mathcal T(i,j)}d_{\mathcal T(i,j)}(u+c(i,j))
$$
when $\mathcal T=\mathcal T_0$. The statement now follows from the fact that $d_i^{(r)}$, $i\in \bar I$ and $r\in\Z_{>0}$, are algebraically independent.
\end{proof}

\subsection{Rational form of quantum Berezinian}
Motivated by \cite{HMVY,HLM} and \cite[Corollary 6.13]{LM19}, we are interested in writing $\mathrm{Ber}(1-T(u)\tau)$ as a ratio of two polynomials in $\tau$.

Let $\Xi$, $\Xi^+$, and $\Xi^-$ be partitions corresponding to the rectangular Young diagrams of sizes $m\times n$, $m\times (n+1)$, and $(m+1)\times n$, respectively. Introduce partitions 
$$
\Upsilon_i^+=(\underbrace{n,\dots,n}_{m~n'\text{s}},i),\qquad  \Upsilon_j^-=(\underbrace{n+1,\dots,n+1}_{j~(n+1)'\text{s}},\underbrace{n,\dots,n}_{(m-j)~n'\text{s}}),
$$
where $i=0,1,\dots,n$ and $j=0,1,\dots,m$. In particular, $\Upsilon_0^+=\Upsilon_0^-=\Xi$, $\Upsilon_n^+=\Xi^-$, $\Upsilon_m^-=\Xi^+$. 

Note that $\mathfrak T_{\Xi}(u)\ne 0$ by Corollary \ref{cor transfer nonzero}.

\begin{theorem}\label{thm ratio}
We have
\begin{align*}
\mathfrak D(u,\tau;q)=&\ \Big(1+\sum_{i=1}^m (-1)^iq^i\mathscr E_i(u)\tau^i\Big)\Big(1+\sum_{j=1}^n q^j\mathscr  G_j(u)\tau^j\Big)^{-1},\\= &\ \Big(1+\sum_{j=1}^nq^j \overline{\mathscr  G}_j(u)\tau^j\Big)^{-1}
\Big(1+\sum_{i=1}^m (-1)^iq^i\overline{\mathscr E}_i(u)\tau^i\Big).
\end{align*}
where
\[
\mathscr E_i(u)=\frac{\mathfrak T_{\Xi^+/(1^{m-i})}(u+m-i)}{\mathfrak T_{\Xi}(u+m+1-i)},\qquad \mathscr G_i(u)=\frac{\mathfrak T_{\Upsilon_i^+}(u+m+1-i)}{\mathfrak T_{\Xi}(u+m+1-i)},
\]
\[
\overline{\mathscr E}_i(u)=\frac{\mathfrak T_{\Upsilon_i^-}(u-n)}{\mathfrak T_{\Xi}(u-n)},\qquad \overline{\mathscr G}_i(u)=\frac{\mathfrak T_{\Xi^-/(m-i)}(u-n+1)}{\mathfrak T_{\Xi}(u-n)},
\]
are ratios of transfer matrices. Here we can take ratio as transfer matrices commute. 
\end{theorem}

\begin{proof}
We only show the first equality. The second one is similar.

By Theorem \ref{thm ber expansion} and Proposition \ref{prop anti}, it suffices to show that
\[
\Big(1+\sum_{k=1}^\infty (-1)^kq^k\mathfrak T^k(u)\tau^k\Big)\Big(1+\sum_{j=1}^n q^j\mathscr  G_j(u)\tau^j\Big)=1+\sum_{i=1}^m (-1)^iq^i\mathscr E_i(u)\tau^i.
\]This reduces to show that
\begin{equation}\label{eq e_i 1}
\mathscr E_i(u)=\mathfrak T^i(u)+\sum_{a=1}^{\min(i,n)}(-1)^a\mathfrak T^{i-a}(u)\mathscr G_a(u-i+a),\qquad i=1,\dots,m;
\end{equation}
\begin{equation}\label{eq e_i 2}
0=\mathfrak T^j(u)+\sum_{a=1}^{\min(j,n)}(-1)^a\mathfrak T^{j-a}(u)\mathscr G_a(u-j+a),\qquad j=m+1,m+2,\dots.
\end{equation}

Let us first show equation \eqref{eq e_i 1}. By Theorem \ref{thm Jacobi T}, we have
\beq\label{eq det 1}
\mathfrak T_{\Xi^+/(1^{m-i})}(u+m-i)=\begin{vmatrix}
\mathfrak T^{i}(u) & \mathfrak T^{m+1}(u+m-i+1) & \cdots & \mathfrak T^{m+n}(u+m-i+n)\\ 
\mathfrak T^{i-1}(u) & \mathfrak T^{m}(u+m-i+1) & \cdots & \mathfrak T^{m+n-1}(u+m-i+n)\\
\vdots & \vdots & \ddots & \vdots \\
\mathfrak T^{i-n}(u) & \mathfrak T^{m+1-n}(u+m-i+1) & \cdots & \mathfrak T^{m}(u+m-i+n)
\end{vmatrix}
\eeq
and
\beq\label{eq det 2}
\mathfrak T_{\Xi}(u+m+1-i)=\begin{vmatrix}

\mathfrak T^{m}(u+m-i+1) & \cdots & \mathfrak T^{m+n-1}(u+m-i+n)\\
 \vdots & \ddots & \vdots \\
 \mathfrak T^{m+1-n}(u+m-i+1) & \cdots & \mathfrak T^{m}(u+m-i+n)
\end{vmatrix}.
\eeq
It follows from Theorem \ref{thm Jacobi T} that $\mathfrak T_{\Upsilon_a^+}(u+m+1-i)$ is equal to the minor of the matrix in \eqref{eq det 1} obtained by deleting the first column and the $(a+1)$-th row.
Expanding the determinant in \eqref{eq det 1} with respect to the first column and dividing both sides by the determinant in \eqref{eq det 2}, then equation \eqref{eq e_i 1} follows from \eqref{eq det 2}.

Denote by $\mathfrak X_i(u)$ the determinant in \eqref{eq det 1}.

Equation \eqref{eq e_i 2} is proved similarly for $i=m+1,\dots,m+n$. In this case we have $\mathfrak X_i(u)=0$ as the first and $(i+1-m)$-th columns coincide. Hence equation \eqref{eq e_i 2} is obtained by using Theorem \ref{thm Jacobi T} and expanding $\mathfrak X_i(u)=0$ with respect to the first column.

Finally, we show equation \eqref{eq e_i 2} for $i=m+n+k$ where $k\gge 1$. Let $\la$ and $\mu$ be partitions corresponding to rectangular Young diagrams of sizes $(m+k)\times (n+1)$ and $(k-1)\times n$, respectively. Clearly, $\la/\mu$ contains a rectangle of size $(m+1)\times (n+1)$. It follows from Corollary \ref{cor zero T} that $\mathfrak T_{\la/\mu}(u)=0$. Applying Theorem \ref{thm Jacobi T}, we have
\[
0=(-1)^n\mathfrak T_{\la/\mu}(u-n)=\mathfrak X_i(u).
\]Again, equation \eqref{eq e_i 2} is obtained by expanding the determinant $\mathfrak X_i(u)$ with respect to the first column.
\end{proof}

\begin{rem}
In the even case $n=0$, the first equality in Theorem \ref{thm ratio} takes the form 
\beq\label{eq:neww}
\mathfrak D(u,\tau;q)=1+\sum_{i=1}^m (-1)^iq^i\mathfrak T_{(1^m)/(1^{m-i})}(u+m-i)\tau^i.
\eeq
Note that the skew Young diagram $(1^m)/(1^{m-i})$ is exactly the column Young diagram $(1^i)$ shifted down by $m-i$ (namely, the contents are shifted by $-m+i$). By \eqref{eq:shift-content-T}, we have $\mathfrak T_{(1^m)/(1^{m-i})}(u+m-i)=\mathfrak T^i(u)$. Therefore, using Proposition \ref{prop anti}, the equality \eqref{eq:neww} reduces back to Theorem \ref{thm ber expansion} with $n=0$.

We note that it does not provide a new proof of Theorem \ref{thm ber expansion} since Theorem \ref{thm ber expansion}  was used in the proof of Theorem \ref{thm ratio}.
\end{rem}

Recall that $\C^{(p)}_{1-\frac{1}{u}}$ is the one-dimensional parity $p$ $\YglMN$-module generated by a vector of highest $\ell$-weight $(1-\frac{1}{u},\dots,1-\frac{1}{u})^p$ and that we write $\C_{1-\frac{1}{u}}$ for $\C_{1-\frac{1}{u}}^{(\bar 0)}$. 

\begin{corollary}\label{cor:Ber}
We have 
$$
\mathfrak Z(u)=(-1)^n\frac{\mathfrak T_{\Xi^+}(u)}{\mathfrak T_{\Xi^-}(u+1)}=\mathfrak T_{\C_{1-\frac{1}{u}}}.
$$
\end{corollary}
\begin{proof}
Note that $(-q)^{n-m}\mathrm{Ber}(1-q\, T(u)\tau)\to \mathfrak Z(u)\tau^{m-n}$ as $q\to \infty$. The first equality follows from Theorem \ref{thm ratio} by taking the limit $q\to \infty$. It is easy to see that
\[
L_0\big((\Xi^+)^\natural\big)\cong L_{-1}\big((\Xi^-)^\natural\big)\otimes \C^{(\bar n)}_{1+\frac{1}{u}}.
\]The second equality follows since $\mathfrak T$ is a homomorphism of rings.
\end{proof}

The first equality of Corollary \ref{cor:Ber} also appears in \cite[equation (2.54)]{MNV}.

\subsection{Spectra of transfer matrices and divisibility of $q$-characters}\label{sec spectra} In this section, we describe the relation between Theorem \ref{thm ratio} and the results in \cite{HLM}. 

Let $M$ be a finite-dimensional irreducible $\YglMN$-module of highest $\ell$-weight $\bm \zeta=(\zeta_i(u))_{i\in \bar I}^s$. We are interested in finding the spectra of transfer matrices acting on the space $M$.

Let $\bm l=(l_i)_{i\in I}$ be a sequence of non-negative integers. Let $ \bm t=(t_{j}^{(i)})$, $i\in I$, $j=1,\dots,l_i$,  be a sequence of complex numbers. Define monic polynomials 
$$
y_i(u)=\prod_{j=1}^{l_i}(u-t_j^{(i)}),
$$
and set $\bm y=(y_i)_{i\in I}$.

The {\it Bethe ansatz equation} associated to $\bm \zeta$, $\bm l$, $(s_i)_{i\in \bar I}$ is a system of algebraic equations in $\bm t$ given by
\beq\label{eq:bae}
\frac{\zeta_i(t_j^{(i)})}{\zeta_{i+1}(t_j^{(i)})}\frac{y_{i-1}(t_j^{(i)}+s_i)}{y_{i-1}(t_j^{(i)})}\frac{y_{i}(t_j^{(i)}-s_i)}{y_{i}(t_j^{(i)}+s_{i+1})}\frac{y_{i+1}(t_j^{(i)})}{y_{i+1}(t_j^{(i)}-s_{i+1})}=1,
\eeq
where $i\in I$ and $j=1,\dots,l_i$. It is known that when the Bethe ansatz equation is satisfied, one can construct the {\it Bethe vector} $\mathbb B_{\bm l}(\bm t)\in M$ which is shown to be an eigenvector (if it is nonzero) of the first transfer matrix $\str(T(u))$, see \cite{BR08}. One also expects the Bethe vector to be an eigenvector of all transfer matrices. Motivated by \cite{HLM}, we have the following conjecture. 

Let $y_0(u)=y_{m+n}(u)=1$. Define a rational difference operator $\mathfrak D(u,\tau,\bm \zeta,\bm y;q)$ by
\beq\label{eq bae eigenvalue}
\mathfrak D(u,\tau,\bm \zeta,\bs y;q)=\mathop{\overrightarrow\prod}\limits_{1\lle i\lle m+n}\Big(1-q\,\zeta_i(u)\cdot\frac{y_{i-1}(u+s_i)y_i(u-s_i)}{y_{i-1}(u)y_i(u)}\, \tau\Big)^{s_i}.
\eeq

\begin{conj}\label{conj bae}
If $\bm t$ satisfies the Bethe ansatz equation \eqref{eq:bae}, then we have
$$
\mathfrak D(u,\tau;q)\, \mathbb B_{\bm l}(\bm t)=\mathfrak D(u,\tau,\bm \zeta,\bs y;q)\, \mathbb B_{\bm l}(\bm t).
$$
\end{conj}
The conjecture was confirmed for $\YglN$ in \cite[Theorem 6.1]{MTV06} and for $\mathrm{Y}(\gl_{1|1})$ in \cite[Theorem 6.5]{LM19}. Conjecture \ref{conj bae} can be thought as the supersymmetric version of \cite[Theorem 5.11]{FH:2015} and \cite[Theorem 7.5]{FJMM17}. Namely, the eigenvalues of transfer matrix associated to a finite dimensional $\YglMN$-module $W$ acting on the finite dimensional $\YglMN$-module $M$ can be obtained by applying certain substitutions to the Harish-Chandra image of $\mathfrak T_{W}(u)$. For instance, applying the substitutions
\beq\label{eq sub spectra}
d_i(u)\mapsto \zeta_i(u)\cdot\frac{y_{i-1}(u+s_i)y_i(u-s_i)}{y_{i-1}(u)y_i(u)}
\eeq
to the Harish-Chandra image $\mathscr H(\mathfrak D(u,\tau;q))$ in Corollary \ref{cor hc image ber}, one obtains exactly the rational difference operator $\mathfrak D(u,\tau,\bm \zeta,\bs y;q)$ in \eqref{eq bae eigenvalue}.

The rational difference operator on the right hand side of \eqref{eq bae eigenvalue} can also be understood using Theorem \ref{thm ratio} and the divisibility of $q$-characters in Section \ref{sec div}. Let
\[
\mathfrak D_1(u,\tau,\bm \zeta,\bm y;q)=\mathop{\overrightarrow\prod}\limits_{1\lle i\lle m}\Big(1-q\,\zeta_i(u)\cdot\frac{y_{i-1}(u+s_i)y_i(u-s_i)}{y_{i-1}(u)y_i(u)}\, \tau\Big), 
\]
\[
\mathfrak D_2(u,\tau,\bm \zeta,\bm y;q) = \mathop{\overleftarrow\prod}\limits_{m+1\lle i\lle m+n}\Big(1-q\,\zeta_i(u)\cdot\frac{y_{i-1}(u+s_i)y_i(u-s_i)}{y_{i-1}(u)y_i(u)}\, \tau\Big).
\]
Then 
\beq\label{eq ratio decom scalar}
\mathfrak D(u,\tau,\bm \zeta,\bs y;q)= \mathfrak D_1(u,\tau,\bm \zeta,\bs y;q)\big(\mathfrak D_2(u,\tau,\bm \zeta,\bs y;q)\big)^{-1}.
\eeq
The rational form decomposition of $\mathfrak D(u,\tau;q)$ in Theorem \ref{thm ratio} is consistent with that of $\mathfrak D(u,\tau,\bm \zeta,\bm y;q)$ in \eqref{eq ratio decom scalar} in the following sense. We compute the Harish-Chandra images of $\mathscr E_i(u)$ in Theorem \ref{thm ratio}. Consider $\mathscr H(\mathfrak T_{\la/\mu}(u))$ as a polynomial in formal variables $d_i(u+a)$ for $i \in \bar I$ and $a\in \C$, see Proposition \ref{prop hc image skew}. It follows from Lemma \ref{lem character q-div} and Proposition \ref{prop hc image skew} that the polynomial $\mathscr H(\mathfrak T_{\Xi}(u+m+1-i))$ divides the polynomial $\mathscr H(\mathfrak T_{\Xi^+/(1^{m-i})}(u+m-i))$ and the quotient is 
\[
\mathscr H(\mathscr E_i(u))=\sum_{1\lle j_1<\dots <j_i\lle m}\prod_{a=1}^i d_{j_a}(u-a+1).
\]
Similarly, we have
\[
\mathscr H(\mathscr G_i(u))=\sum_{1\lle j_1<\dots <j_i\lle n}(-1)^i\prod_{a=1}^i d_{m+j_a}(u-i+a).
\]Applying substitutions \eqref{eq sub spectra} again, we have
\[
\mathscr H\Big(1+\sum_{i=1}^m (-1)^iq^i\mathscr E_i(u)\tau^i\Big)\mapsto \mathfrak D_1(u,\tau,\bm \zeta,\bs y;q),\quad 
\mathscr H\Big(1+\sum_{j=1}^n q^j\mathscr  G_j(u)\tau^j\Big)\mapsto \mathfrak D_2(u,\tau,\bm \zeta,\bs y;q).
\]

\begin{rem}
There is a similar explanation for the second rational form decomposition of $\mathfrak D(u,\tau;q)$ in Theorem \ref{thm ratio} using the divisibility of $q$-characters in Section \ref{sec div}. However, in this case, it is known from \cite{HMVY,HLM} that the corresponding parity sequence is $(-1,\dots,-1,1,\dots,1)$ (the number $-1$ occurs $n$ times while  $1$ occurs $m$ times) instead of the standard parity sequence $(1,\dots,1,-1,\dots,-1)$. For the semi-standard Young tableaux, we fill in numbers $1,\dots,n$ strictly increasing along rows and weakly increasing along columns while $n+1,\dots,n+m$ are filled in strictly increasing along columns and weakly increasing along rows. Then it is easy to show similar divisibility equalities of $q$-characters (Lemma \ref{lem character q-div}) for modules corresponding to the skew Young diagrams involved in $\overline{\mathscr E}_i(u)$ and $\overline{\mathscr G}_i(u)$ from Theorem \ref{thm ratio}.
\end{rem}

\end{document}